\documentclass[11pt]{article}
\usepackage{amsmath}
\usepackage{amssymb}
\usepackage{amsthm}
\usepackage{latexsym}
\usepackage{enumitem}
\usepackage{color}
\usepackage{graphicx}
\usepackage{color}
\usepackage[pdftex=true, colorlinks=true,linkcolor=blue,citecolor=magenta]{hyperref} 
\DeclareSymbolFont{calletters}{OMS}{cmsy}{m}{n}
\DeclareSymbolFontAlphabet{\mathcal}{calletters}
\newtheorem{Theorem}{Theorem}[part]
\newtheorem{Definition}{Definition}[part]
\newtheorem{Proposition}{Proposition}[part]
\newtheorem{Assumptions}{Assumptions}[part]
\newtheorem{Lemma}{Lemma}[part]

\def\reff#1{{\rm(\ref{#1})}}

\def\Ac{{\cal A}}

\def\Fc{{\cal F}}

\def\no{\noindent}

\def\05{\frac{1}{2}}
\def\-1{^{-1}}
\def\1{{1\hspace{-1mm}{\rm I}}}
\def\={\;=\;}
\def\.{\;.}

\title{  }
\author{ }

\def\be{\begin{eqnarray}}
\def\ee{\end{eqnarray}}

\def\b*{\begin{eqnarray*}}
\def\e*{\end{eqnarray*}}

\def\E{\mathbb{E}}
\def\F{\mathbb{F}}

\def\R{\mathbb{R}}

\def\P{\mathbb{P}}

\newcommand{\nc}{\newcommand}
\nc{\argmin}{\mathop{\mathrm{arg\,min}}}

\def \F{I\!\!F}

\def \R{I\!\!R}

\def\Ac{{\cal A}}

\def\Fc{{\cal F}}

\def\-1{^{-1}}

\def\0.5{\frac{1}{2}}

\def\no{\noindent}
\def\={\;=\;}
\def\.{\;.}

\def\reff#1{{\rm(\ref{#1})}}

\def\1{{\bf 1}}

\newcommand{\MBFigure}[6]{
$\left. \right.$ \\
\refstepcounter{figure}
\addcontentsline{lof}{figure}{\numberline{\thefigure}{\ignorespaces #5}}
\begin{center}
\begin{minipage}{#1cm}
\centerline{\includegraphics[width=#2cm,angle=#3]{#4}}
\begin{center}
\upshape{F\textsc{ig} \normal
\end{center}
size{\thefigure}. $-$} #5
\end{center}
\label{#6}
\end{minipage}
\end{center}
$\left. \right.$ \\}
\title{Mean-Field Backward-Forward SDE with Jumps and Storage problem in Smart Grids
}
\author{
Arij {\sc Manai}
\\
	\texttt{ arij.manai@univ-lemans.fr} \footnote{Le Mans Universit\'e, Institut  du Risque et de l'Assurance du Mans.}
\and
Anis {\sc Matoussi}\\
	\texttt{ anis.matoussi@univ-lemans.fr} \footnote{Le Mans Universit\'e, Institut  du Risque et de l'Assurance du Mans.}
   \and
   Rym {\sc Salhi}
    \\
\texttt{ rym.salhi@univ-lemans.fr}
 \footnote{Le Mans Universit\'e, Institut  du Risque et de l'Assurance du Mans.}
  }
\begin{document}
\maketitle
\begin{abstract}
In this paper, we prove the existence and uniqueness of the solution of a coupled Mean-Field Forward-Backward SDE system with Jumps. Then, we give an application in the field of storage problem in smart grids, studied in \cite{clemence2017grid} in the case where the production of electricity is not predictable due, for example, to the changes in meteorological forecasts. 

\end{abstract}
\section{Introduction}
\paragraph{Forward Backward SDEs with jumps}
In this paper we focus on Mean field coupled forward BSDE (FBSDEs in short) with jumps. The theory of Forward Backward SDEs was first introduced in the continuous framework by Antonelli \cite{antonelli1993backward} in his PhD thesis where he proved the well-posedness of these equations over a sufficiently small time duration. The author also argue through a counter example that, over a considerable time period, the solvability of the Forward Backward SDEs may fail. Ma, Protter and Young \cite{MaProtterYoung4step} proved that the FBSDE is linked to some quasilinear PDE under the condition that the Forward equation satisfies the non degeneracy condition and the coefficients are not randomly disturbed. Later, Hu and Peng \cite{HUPeng95} proved existence and uniqueness under an arbitrary duration time by relaxing the degeneracy condition and substitute it by a monotonicity condition. In \cite{PengWu99}, Wu and Peng treated the same problem and weakened the monotonicity condition. They also make a link with linear quadratic stochastic optimal control problem using the maximum principle. We also refer to the paper of Yu and Ji \cite{YU2008} who studied the linear quadratic non zero sum games in term of FBSDE.\\

This theory has been extensively studied in the literature and many extensions have been made. In this paper, we are concerned with its extension to the jump setting for which the literature is very small. The first generalization of fully coupled FBSDEs to the jump setting was obtained by Zen Wu \cite{zhen1999forward}. Importantly,in \cite{zhen1999forward} and \cite{wu2003fully}, Wu extended the results in \cite{yong1997finding} and \cite{peng1999fully} to the case where the FBSDE is driven by both a Brownian motion and a Poisson random measure under the monotonicity condition. More precisely, in \cite{zhen1999forward}, the author obtained the existence and the uniqueness of the solution for such fully coupled FBSDEs with jumps, and in \cite{wu2003fully}, he proved the existence and the uniqueness of the solution as well as a comparison theorem for fully coupled FBSDEs with jumps over a stochastic interval.
\paragraph{Mean field forward-backward SDEs theory}
Historically the theory of mean field stochastic differential equations, also called \textit{McKean-Vlasov equations}, goes back to the early works of Kac \cite{kac1956foundations} and McKean \cite{mckean1966class} in the 1950's. It was initially suggested in order to study the behavior of a large number of mutually-interacting particles in different fields of of physical science. e.g. the derivation of Boltzmann or Vlasov equations in the kinetic gas theory.

With their pioneering recent paper \cite{lasry2007mean} , Lasry and Lions have enlarged considerably the horizon for applications of mean field problems. They extended this approach to problems in economics, finance and also to the theory of stochastic differential games where they introduced a general mathematical modeling approach of situations where a large number of particles is involved. More precisely, the authors considered N-player stochastic differential games, proved the existence of the Nash equilibrium points and derived rigorously the mean-field limit equation as N goes to infinity. Doing so, simple forms of BSDEs of McKean Vlasov type have been introduced and called of mean field type. 

In $2009$, nonlinear mean-field backward stochastic differential equations had been investigated with the work of Buckdahn, Djehiche Li and Peng \cite{buckdahn2009mean}. Since then, the theory of mean-field forward-backward stochastic differential equations(FBSDEs in short), as well as the theory of the associated partial differential equations (PDEs in short) of mean-field type has been intensively studied in the literature.

Buckdahn, Li and Peng \cite{buckdahn2009mf} studied a mean field problem in a Markovian setting. More precisely,  the authors investigate  in one hand, the existence and uniqueness of the mean field BSDEs in more general setting than the one in \cite{buckdahn2009mean}. In fact, unlike in \cite{buckdahn2009mean}, they consider that the coefficients are not necessarily deterministic. In the other hand, they give a comparison principle or this new type of BSDEs and study a decoupled mean-field FBSDEs and its relation with PDEs.

\paragraph{Contributions}
Our first purpose in this paper is to obtain existence and uniqueness results for a general class of fully coupled backward-forward stochastic differential equations of mean-field type (MF-FBSDE in short) with jumps under weak monotonicity conditions and without the non-degeneracy assumption on the forward equation.

Later, we give an application for storage in smart grids. This part is motivated by results of Alasseur et al. \cite{clemence2017grid}, where the authors  provide a stylized quantitative model for a power system with distributed local energy generation and storage.

This system is modeled as a network connecting a large number of nodes.  Each node has a local electricity consumption, a local electricity production, and manages a local storage device as in \cite{clemence2017grid}. However, unlike \cite{clemence2017grid}, we assume that the production of energy is unpredictable due to its dependence on environmental conditions such as the sun, the speed of the wind etc. which are intermittent and irregular. This leads to include a jump component in the net power production of each node. \\

 We highlight that an important volume of literature is devoted to the distributed storage management and the analysis of its development within the system. In particular, mean field games (MFG)approach has been already used by \cite{couillet2012mean} who analyze a system with controlled electrical vehicles and by \cite{de2015distributed} with local batteries. These two papers deal with numerical analysis of a corresponding MFG without providing the existence and uniqueness of the optimal control results.
 
 We also would like to stress out that the theory of mean field type control (MFC in short) is different from the mean field game (MFG in short) and even though, in general an optimal control on MFC is not an equilibrium strategy on MFG, nevertheless Lasry and Lions in \cite{lasry2007mean} have pointed out that in many cases a mean field Nash equilibrium is also the solution to an optimal control problem. Moreover, motivated by economic applications, Graber \cite{Graber2016}, have highlighted this point of view. \\
\noindent The outline of this paper is as follows. After recalling briefly some notations, we define, in section \ref{section 2}, the system of fully coupled forward backward SDE with jumps and we suggest existence and uniqueness results under different assumptions $\textbf{(H1)}$ and $\textbf{(H2)}$.
 
\noindent Then, in section \ref{section 3-4}, we consider a stylized model for a power network with distributed local power generation and storage. This model has been considered in \cite{clemence2017grid} where the system is modeled as network connecting a large number of nodes, where each node is characterized by a local electricity consumption, has a local electricity production  and manages a local storage device. In this part, in contrast to \cite{clemence2017grid},  we take account of the unpredictability of the energy production. In fact, is depends on intermittent and irregular environmental conditions and meteorological forecasts.\\
\noindent To illustrate this phenomena, in section \ref{section 3-4}, we include a jump component in our analysis. When the number of nodes is infinite, we link the problem to  an extended mean field game type control which unique solution is characterized through solving an associated fully forward backward stochastic differential equations with jumps.

\section{Framework: Notations and setting}
Let  $(\Omega ,\mathcal{F},\mathbb{P})$ be a filtered probability space on which is defined a $d$-dimensional Brownian motion $W=(W_{t})_{0\leq t\leq T}$ and an integer valued random measure $\pi$. We assume that the filtration $\mathbb{F}=(\mathcal{F}_{t})_{0<t<T}$ satisfies the usual conditions of completeness and right continuity.\\
Here $\pi$ is an integer valued random measure defined as
\begin{align}
\pi(\omega,dt,de):&(\Omega\times[0,T]\times\mathcal{E})\rightarrow(\mathcal{B}([0,T])\times \mathcal{E})  \nonumber \\
                  &(w,dt,de)\rightarrow \pi(\omega,dt,de)=\sum^{\Delta X_{s}\neq 0}_{s\in[0,T]}\delta_{ (s,\Delta X_{s}) }(dt,de).\nonumber
\end{align}
We will denote by $\eta$ the compensator of $\pi$ under the probability $\mathbb{P}$ (increasing process in any measurable subset $A$ and compensates the number of jump of $X$)  . For a $\sigma$-finite measure $\lambda$ and a bounded Radon-Nikodym derivative $\zeta$, we will assume that the compensator $\nu$ is absolutely continuous with respect to $\lambda\otimes dt$ such that
$$ \eta(dt,de)=\zeta(\omega,t,e)\lambda (de)dt, \,\,\,\,\,\,\,\,\,\,\,\, 0\leq \zeta \leq C_{\nu},$$
and  satisfies the following integrability condition $\int_{E}1\wedge |e|^{2}\lambda (de)<\infty$. We emphasize that, in our setting, jumps has finite activity jumps i.e.  $\int_{E} \eta(de)<\infty$.

\noindent The random measure $\tilde{\pi}$ is defined as the compensated measure of $\mu$ such that
$$ \tilde{\pi}(w,dt,de)=\pi(w,dt,de)-\eta(dt,de).$$

\noindent For any random variable $X$ on $(\Omega, \mathcal{F}_{t}, \mathbb{P})$, we denote by $\mathbb{P}_{X}$ its probability law under $\mathbb{P}$.
We denote by $\mathcal{M}_{1}(\mathbb{R}^{d})$ the set of probability measures on $\mathbb{R}^{d}$ equipped with the $2$-Wassertein distance 
\begin{align*}
\mathcal{W}_{2}(\mu, \mu')&:=\{(\int_{\mathbb{R}^{d}}|x-y|^{2}F(dx,dy))^{\frac{1}{2}}, F \in \mathcal{M}_{1}(\mathbb{R}^{d}\times \mathbb{R}^{d}) \mbox{ with marginals } \mu, \mu^{'}\}\\
&:=\{(\mathbb{E}|\xi - \xi'|^{2})^{\frac{1}{2}}: \xi=\mathcal{L}(\xi), \xi^{'}=\mathcal{L}(\xi^{'})\},
\end{align*}
where $\mathcal{L}(\xi)$ and $\mathcal{L}(\xi^{'})$ are respectively the law of $\xi$ and $\xi^{'}$.\\

\noindent Notice that if $X^{1}$ and $X^{2}$ are random variables of order $2$ with values in $\mathbb{R}^{d}$, then by definition we have 
\begin{equation}\label{Wasserstein}
\mathcal{W}_{2}(\mathbb{P}_{X^{1}}, \mathbb{P}_{X^{2}})\leq \Big[\mathbb{E}|X^{1}-X^{2}|^{2}\Big]^{\frac{1}{2}}.
\end{equation}
\noindent Now, we define the spaces of processes which will be used in the present work. We will denote by $\mathcal{P}$ the $\sigma$-field on $\Omega \times [0,T]$ generated by all left continuous adapted processes.


We denote by 
\begin{itemize}
\item $\mathcal{L}^{2}$ is the space of all $\mathbb{R}^{k}$-valued RCLL and $\mathcal{P}$-measurable such that 
$$
\E\int_{0}^{T} |Z_{s}|^{2}ds<+\infty,  \,\,\mathbb{P}\hbox{-a.s.}
$$
\item $\mathcal{L}^{2}_{\eta}$ is the space of all predictable processes such that$$  \|U\|^{2}_{\mathcal{L}^{2}_{\eta}}:=\E\int_{0}^{T}\int_{\mathbb{E}}|U_{s}(e)|^{2}\eta(de,ds)<+\infty, \,\,\mathbb{P}\hbox{-a.s.}$$
\item  For $u,\bar{u}\in\mathbb{L}^{0}(E,\lambda, \R^{d})$, we define 
\begin{equation*}
| u-\bar{u}|^{2}_{t}=\int_{E}|u-\bar{u}|^{2}\eta(de).
\end{equation*}
\end{itemize}


\section{The system of mean-field FBSDE's with jumps}
\label{section 2}
In this section, we study the solvability of the following fully coupled forward backward SDE with jumps of mean field type
\begin{equation*}
\label{FBSDE}
\textbf{(S)}
{\small{
\begin{cases}
\vspace{0.1cm}
&\displaystyle{X_{t}=X_{0}+\int_{0}^{t}b_{s}(X_{s},Y_{s},Z_{s},K_{s},\mathbb{P}_{(X_{s},Y_{s})})ds+\int_{0}^{t}\sigma_{s}(X_{s},Y_{s},Z_{s},K_{s},\mathbb{P}_{(X_{s},Y_{s})}))dW_{s}}\\
\vspace{0.1cm}
& \qquad \displaystyle{+\int_{0}^{t}\int_{E}\beta(s,X_{s^{-}},Y_{s^{-}},Z_{s},K_{s},\mathbb{P}_{(X_{s},Y_{s})})\tilde{\pi}(ds,de),\hspace{0.2cm}0 \leq t\leq T,\, \P}\hbox{-a.s}.\\
&\displaystyle{Y_{t}=g(X_{T},\mathbb{P}_{X_{T}})+\int_{t}^{T}h_{s}(X_{s},Y_{s},Z_{s},K_{s},\mathbb{P}_{(X_{s},Y_{s})})ds-\int_{t}^{T}Z_{s}dW_{s}-\int_{t}^{T}\int_{E}K_{s}(e)\tilde{\pi}(ds,de).}
\end{cases}
}}
\end{equation*}
where $W$ is a $d$-dimensional Brownian motion, $(X,Y,Z,K)$ is an $\mathbb{R}^{d}\times \mathbb{R}^{d}\times \mathbb{R}^{d\times d}$-valued adapted processes and $\mathbb{P}_{(X_{t},Y_{t})}$ is the marginal distribution of $(X_{t},Y_{t})$.
  \noindent We require that the coefficients of the system $(\textbf{S})$ satisfy the following assumptions.
 \begin{Assumptions}
 \textbf{Lipschitz Assumptions}
 \begin{itemize}
   \item[1-] The functions $b$, $h$, $\sigma$ and $\beta$ are Lipschitz in $(x,y,z,k)$ i.e.  there exists a constant $C>0$ such that $\forall t\in[0,T]$, $u=(x,y,z,k)$, $u'=(x',y',z',k') \in \mathbb{R}^{d+d+d\times d}$ and $\nu, \nu' \in \mathbb{M}_{1}(\mathbb{R}^{d}\times \mathbb{R}^{d})$,
\begin{align}\label{Lip of b sigma beta h }
&|b(t,u,\nu)-b(t,u',\nu')|+|h(t,u,\nu)-h(t,u',\nu')|+|\sigma(t,u,\nu)-\sigma(t,u',\nu')|\\
&+|\beta(t,u,\nu)-\beta(t,u',\nu')|\leq C\Big(|x-x'|+|y-y'|+\|z-z' \|+|k-k'|_{t}+\mathcal{W}_{2}(\nu,\nu')\Big).\nonumber
\end{align}
\item[2-] For $\phi\in\left\lbrace f,h,g,\sigma\right\rbrace, 
$ $\phi$ is Lipschitz w.r.t $x, y, z, \nu$ with $C^{x}_{\phi}$, $C^{y}_{\phi}$, $C^{z}_{\phi}$ and  $C^{\nu}_{\phi}$ as the Lipschitz constants. We also assume that $\beta$ satisfies the following conditions
\begin{equation}
\label{saut-lipschitz}
|\beta(x,\nu,e)|\leq C^{x}_{\beta}(1\wedge|e|)\quad\mbox{ and } \quad |\beta(x,\nu,e)-\beta(x',\nu',e)|\leq C^{x}_{\beta}(1\wedge |e|)|x-x'|.
\end{equation}
\item[3-] The function  $g:\Omega\times\mathbb{R}^{d}\times\mathbb{M}_{1}(\mathbb{R}^{d})\rightarrow \mathbb{R}^{d}$  is Lipschitz in $(x,\mu)$ i.e. there exists $C>0$ such that for all $x,\,\, x' \in \mathbb{R}^{d}$ and for all $\mu,\,\, \mu' \in \mathbb{M}_{1}(\mathbb{R}^{d})$,
\begin{equation}
\label{g-Lipschitz}
|g(x,\mu)-g(x',\mu')|\leq C(|x-x'|+\mathcal{W}_{2}(\mu,\mu')),\,\, \mathbb{P}\hbox{-a.s}.
\end{equation}
\end{itemize}
\end{Assumptions}

\noindent For $u=(x,y,z,k)$ and $u'=(x',y',z',k') \in \mathbb{R}^{d+d+d\times d}\times \mathbb{L}^{0}(E,\lambda, \mathbb{R}^{d})$, we define the operator $\mathcal{A}$ in the following way
\begin{align*}
\mathcal{A}(t,u,u',\nu)&=\big[f(s,u,\nu)-f(s,u',\nu')\big](y-y')+\big[h(s,u,\nu)-h(s,u',\nu')\big](x-x')\\
&+[\sigma(s,u,\nu)-\sigma(s,u',\nu')](z-z')\\
&+\int_{E}(\beta(s,u,\nu)-\beta(s,u',\nu'))(k-k')(e)\eta(ds,de),
\end{align*} 
\subsection{Existence and uniqueness under $\textbf{(H1)}$}
In this part, we will present our first result which consists on proving the existence and uniqueness of the solution of the system $\textbf{(S)}$ under the following assumption:
\begin{equation*}\textbf{(H1)}\,\,\,
\label{h1-assump}
\begin{cases} 
(i)\mbox{ There exists } k>0, \mbox{ s.t } \forall t \in[0,T], \nu\in \mathbb{M}_{1}(\mathbb{R}^{d}\times \mathbb{R}^{d}), u, u' \in \mathbb{R}^{d+d+d \times d}, \,\,\,\,\\
 \qquad \qquad \mathcal{A}(t,u,u',\nu)\leq -k|x-x'|^{2}, \mathbb{P}\hbox{-a.s}.\\
(ii)\mbox{ There exists } k'>0, \mbox{ s.t } \forall \nu\in \mathbb{M}_{1}(\mathbb{R}^{d}\times \mathbb{R}^{d}),x,x'\in \mathbb{R}^{d}\\
\qquad \qquad(g(x,\nu)-g(x',\nu)).(x-x')\geq  k'|x-x'|^{2}, \mathbb{P}\mbox{-a.s.}
\end{cases}
\end{equation*}
We start by giving a key estimate for the difference of two solutions of the mean-field  fully coupled FBSDEs with jumps $\textbf{(S)}$  satisfying $\textbf{(H1)}$.
\begin{Lemma}\label{lemme estimation constante MBSDE}
Let $(Y^{'},Z^{'},K^{'})$ another solution  of the system $\textbf{(S)}$. Then, under $(\textbf{H1})$ we have the following estimates 
 \begin{equation}
 \E[|Y_{s}-Y^{'}_{s}|^{2}]\leq \Theta^{1}\E[|X_{T}-X^{'}_{T}|^{2}]+\Theta^{2}\int_{0}^{T}\E|X_{s}-X^{'}_{s}|^{2}ds
 \end{equation}
 \begin{equation}
 \E[\int_{0}^{T}[|Z_{s}-Z^{'}_{s}|^{2}+|K_{s}-K^{'}_{s}|_{s}^{2}]ds\leq \bar{\Theta}^{1}\E[|X_{T}-X^{'}_{T}|^{2}]+\bar{\Theta}^{2}\int_{0}^{T}\E|X_{s}-X^{'}_{s}|^{2}ds,
 \end{equation}
where
\begin{equation}
\begin{cases}
&\bar{\Theta}^{1}=2[(C_{h}^{x}+(C_{h}^{z})^{2}+(C_{h}^{k})^{2}+2C_{h}^{\nu}+2C_{h}^{y})]\Theta^{1}+2(C^{x}_{g}+C_{g}^{\nu})^{2}\\
&\bar{\Theta}^{2}=2[(C_{h}^{x}+(C_{h}^{z})^{2}+(C_{h}^{k})^{2}+2C_{h}^{\nu}+2C_{h}^{y})]\Theta^{2}+2(C_{h}^{x}+C_{h}^{\nu})\\
      &\Theta^{1}=  e^{(C_{h}^{x}+(C_{h}^{z})^{2}+(C_{h}^{k})^{2}+2C_{h}^{\nu}+2C_{h}^{y})T}(C^{x}_{g}+C_{g}^{\nu})^{2}\\
      &\Theta^{2}=e^{(C_{h}^{x}+(C_{h}^{z})^{2}+(C_{h}^{k})^{2}+2C_{h}^{\nu}+2C_{h}^{y})T}(C_{h}^{x}+C_{h}^{\nu}).
        \end{cases}
    \end{equation}
\end{Lemma}
\begin{proof}
For simplicity, we shall make the following notations that will be used all along this paper: $\Delta{X}=X'-X,\,\, \Delta{Y}=Y'-Y,\,\, \Delta{Z}=Z'-Z,\,\, \Delta{K}=K'-K$. 
Now, consider the following processes
\begin{align*}
 &\displaystyle{\zeta^{1}_{s}=\frac{h(X'_{s},Y'_{s},Z'_{s},K'_{s},\P_{(X'_{s},Y'_{s})})-h(X_{s},Y'_{s},Z'_{s},K'_{s},\P_{(X'_{s},Y'_{s})})}{X^{'}_{s}-X_{s}}},\\
&\zeta^{2}_{s}=\frac{h(X_{s},Y'_{s},Z'_{s},K'_{s},\P_{(X'_{s},Y'_{s})})-h(X_{s},Y_{s},Z'_{s},K'_{s},\P_{(X'_{s},Y'_{s})})}{Y^{'}_{s}-Y_{s}},\\
&\zeta^{3}_{s}=\frac{h(X_{s},Y_{s},Z'_{s},K'_{s},\P_{(X'_{s},Y'_{s})})-h(X_{s},Y_{s},Z_{s},K'_{s},\P_{(X'_{s},Y'_{s})})}{\|Z^{'}_{s}-Z_{s}\|^{2}}(Z^{'}_{s}-Z_{s}),\\
\end{align*}
which are respectively bounded by $C_{h}^{x},  C_{h}^{y}, C_{h}^{z}$ due to the Lipschitz assumption on $h$. 
We apply It\^{o}'s formula to the process $|\Delta{Y}|^{2}$ to obtain
\begin{align*}
&\E[|\Delta{Y}_{t}|^{2}]=\E[g(X'_{T},\P_{X'_{T}})-g(X_{T},\P_{X_{T}})]^{2}+2\E\int_{t}^{T}\Delta{Y}_{s}[\zeta^{1}_{s}\Delta{X}_{s}+\zeta^{2}_{s}\Delta{Y}_{s}+\zeta^{3}_{s}\Delta{Z}_{s}]ds\\
&-\E[\int_{t}^{T}\int_{E}\Big(|\Delta{Y}_{s^{-}}+\Delta{K}_{s}(e)|^{2}-|\Delta{Y}_{s^{-}}|^{2}\Big)\tilde{\pi}(de,ds)-\int_{t}^{T}\|\Delta{Z}_{s}\|^{2}ds]\\&-\E[\int_{t}^{T}\int_{E}\Big(|\Delta{Y}_{s^{-}}+\Delta{K}_{s}(e)|^{2}-|\Delta{Y}_{s^{-}}|^{2}-2|\Delta{Y}_{s^{-}}\Delta{K}_{s}(e)|\Big)\eta(de,ds)-\int_{t}^{T}2\Delta{Y}_{s}\Delta{Z}_{s}dW_{s}]\\
&+2\E[\int_{t}^{T}\Delta{Y}_{s}\big[h(U^{'}_{s},\P_{(X'_{s},Y'_{s})})-h(U^{'}_{s},\P_{(X_{s},Y_{s})})+h(X_{s},Y_{s},Z_{s},K'_{s},\P_{(X^{'}_{s},Y^{'}_{s})})-h(X_{s},Y_{s},Z_{s},K_{s},\P_{(X'_{s},Y'_{s})})\big]ds.
\end{align*}
Since the stochastic integrals are true martingales, we conclude that
\begin{align} 
\E[|\Delta{Y}_{t}|^{2}]+&\E[\int_{t}^{T}\|\Delta{Z}_{s}\|^{2}ds]+\E[\int_{t}^{T}\int_{E}|\Delta{U}_{s}(e)|^{2}\eta(de,ds)]\nonumber\\
&=\E[|g(X'_{T},\P_{X'_{T}})-g(X_{T},\P_{X_{T}})|^{2}]+2\E\int_{t}^{T}\Delta{Y}_{s}[\zeta^{1}_{s}\Delta{X}_{s}+\zeta^{2}_{s}\Delta{Y}_{s}+\zeta^{3}_{s}\Delta{Z}_{s}]ds\nonumber\\
&+\E[\int_{t}^{T}\Delta{Y}_{s}[h(X_{s},Y_{s},Z_{s},K'_{s},\P_{(X'_{s},Y'_{s})})-h(X_{s},Y_{s},Z_{s},K_{s},\P_{(X'_{s},Y'_{s})})]ds]\nonumber\\
&+2\E[\int_{t}^{T}\Delta{Y}_{s}[(U_{s},\P_{(X'_{s},Y'_{s})})-h(U^{'}_{s},\P_{(X_{s},Y_{s})})]ds].
\end{align}
Using the Lipschitz property of $h$, we obtain that 
\begin{align}\label{ito Y2}
\E\big[|\Delta{Y}_{t}|^{2}+\int_{t}^{T}\|\Delta{Z}_{s}\|^{2}ds&+\int_{t}^{T}\int_{E}|\Delta{K}_{s}(e)|^{2}\eta(de,ds)\big]\leq\E[|g(X'_{T},\P_{X'_{T}})-g(X_{T},\P_{X_{T}})|^{2}]\nonumber\\
&+2\E\int_{t}^{T}\Delta{Y}_{s}[C_{h}^{x}|\Delta{X}_{s}|+C_{h}^{y}|\Delta{Y}_{s}|+C_{h}^{z}|\Delta{Z}_{s}|+C_{h}^{k}|\Delta{K}_{s}|_{s}]ds \nonumber\\
&+2\E[\int_{t}^{T}\Delta{Y}_{s}[h(U'_{s},\P_{(X'_{s},Y'_{s})})-h(U'_{s},\P_{(X_{s},Y_{s})})]ds.
\end{align}
Notice that in one hand we have
\begin{align}\label{estim h constante}
&2\Delta{Y}_{s}[h(U'_{s},\P_{(X'_{s},Y'_{s})})-h(U'_{s},\P_{(X_{s},Y_{s})})]\leq 2C_{h}^{\nu}|\Delta Y_{s}|(\sqrt{\E[|\Delta{X}_{s}|^{2}]}+\sqrt{\E[|\Delta{Y}_{s}|^{2}]}),
\end{align}
 and in the other hand we have
\begin{align}\label{estim const}
&2\Delta{Y}_{s}[C_{h}^{x}|\Delta{X}_{s}|+C_{h}^{y}|\Delta{Y}_{s}|+C_{h}^{z}|\Delta{Z}_{s}|+C_{h}^{k}|\Delta{K}_{s}|_{s}]ds\\
&\leq C_{h}^{x}|\Delta{Y}_{s}|^{2}+C_{h}^{x}|\Delta{X}_{s}|^{2}+(C_{h}^{z})^{2}|\Delta{Y}_{s}|^{2}+|\Delta{Z}_{s}|^{2}+(C_{h}^{k})^{2}|\Delta{Y}_{s}|^{2}+|\Delta{K}|^{2}_{s}+2C_{h}^{y}|\Delta{Y}_{s}|^{2}.\nonumber
\end{align}
Moreover, by Young inequality and the Lispchitz property on $g$ we obtain 
\begin{align}\label{estim g constante}
|g(X'_{T},\P_{X'_{T}})-g(X_{T},\P_{X_{T}})|^{2}&=|g(X'_{T},\P_{X'_{T}})-g(X_{T},\P_{X'_{T}})+g(X_{T},\P_{X'_{T}})-g(X_{T},\P_{X_{T}})|^{2}\nonumber\\
&\leq 
|C_{g}^{x}|X_{T}'-X_{T}|+C_{g}^{\nu}\mathcal{W}_{2}(\mu',\mu)|^{2}\nonumber\\
&\leq (C^{x}_{g})^{2}|\Delta{X}_{T} |^{2}+(C_{g}^{\nu})^{2} |\Delta{X}_{T}|^{2}+2C_{g}^{x}C_{g}^{\nu}|\Delta{X}_{T}|\mathcal{W}_{2}(\mu',\mu)\nonumber\\
&\leq(C^{x}_{g})^{2}|\Delta{X}_{T} |^{2}+(C_{g}^{\nu})^{2} |\Delta{X}_{T}|^{2}+C_{g}^{x}C_{g}^{\nu}|\Delta{X}_{T}|^{2}+C_{g}^{x}C_{g}^{\nu}\mathcal{W}_{2}^{2}(\mu',\mu) \nonumber\\
&\leq(C^{x}_{g})^{2}|\Delta{X}_{T} |^{2}+(C_{g}^{\nu})^{2} |\Delta{X}_{T}|^{2}+2C_{g}^{x}C_{g}^{\nu}|\Delta{X}_{T}|^{2}\nonumber\\
&\leq (C^{x}_{g}+C_{g}^{\nu})^{2}|\Delta{X}_{T}|^{2}.
\end{align}
Now, plugging \eqref{estim g constante}, \eqref{estim const} and  \eqref{estim h constante} in \eqref{ito Y2} yields
\begin{align*}
\E[|\Delta{Y}_{t}|^{2}]
&\leq (C^{x}_{g}+C_{g}^{\nu})^{2}\E[|\Delta{X}_{T}|^{2}]+\int_{t}^{T}(C_{h}^{x}+C_{h}^{\nu})\E|\Delta{X}_{s}|^{2}ds \nonumber\\
&+\E[\int_{t}^{T}(C_{h}^{x}+(C_{h}^{z})^{2}+(C_{h}^{k})^{2}+2C_{h}^{\nu}+2C_{h}^{y})|\Delta{Y}_{s}|^{2}ds].
\end{align*}
Finally, Gronwall's lemma yields 
\begin{equation*}
\E[|\Delta{Y}_{s}|^{2}]\leq e^{[(C_{h}^{x}+(C_{h}^{z})^{2}+(C_{h}^{k})^{2}+2C_{h}^{\nu}+2C_{h}^{y})]T} \Big[(C^{x}_{g}+C_{g}^{\nu})^{2}\E[|\Delta{X}_{T}|^{2}]+(C_{h}^{x}+C_{h}^{\nu})\int_{0}^{T}\E|\Delta{X}_{s}|^{2}ds\Big],
\end{equation*}
and we obtain  
 \begin{equation}
 \E[|\Delta{Y}_{s}|^{2}]\leq \Theta^{1}\E[|\Delta{X}_{T}|^{2}]+\Theta^{2}\int_{0}^{T}\E|\Delta{X}_{s}|^{2}ds.
 \end{equation}
 $\bullet$ Let us now prove the second estimates. 
Recalling \eqref{ito Y2} and noting that 
\begin{equation*}
2C_{h}^{z}|\Delta{Y}_{s}||\Delta{Z}_{s}|\leq 2(C_{h}^{z})^{2}| \Delta{Y}_{s}|^{2}+\frac{1}{2}|\Delta{Z}_{s}|^{2} \mbox{ and } 2C_{h}^{k}|\Delta{Y}_{s}||\Delta{K}_{s}|_{s}\leq 2(C_{h}^{k})^{2}| \Delta{Y}_{s}|^{2}+\frac{1}{2}|\Delta{K}_{s}|_{s}^{2}, 
\end{equation*}
we obtain 
\begin{align*}
\frac{1}{2}\E[\int_{t}^{T}[|\Delta{Z}_{s}|^{2}+|\Delta{K}_{s}|_{s}^{2}]ds]
&\leq (C^{x}_{g}+C_{g}^{\nu})^{2}\E[|\Delta{X}_{T}|^{2}]+\int_{t}^{T}(C_{h}^{x}+C_{h}^{\nu})\E|\Delta{X}_{s}|^{2}ds \nonumber\\
&+\E[\int_{t}^{T}(C_{h}^{x}+2(C_{h}^{z})^{2}+2(C_{h}^{k}+C_{h}^{\nu})^{2}+2C_{h}^{y})|\Delta{Y}_{s}|^{2}ds]
\end{align*}
Henceforth,  if we denote by 
\begin{align*}
&\bar{\Theta}^{1}=2[(C_{h}^{x}+(C_{h}^{z})^{2}+(C_{h}^{k})^{2}+2C_{h}^{\nu}+2C_{h}^{y})]\Theta^{1}+2(C^{x}_{g}+C_{g}^{\nu})^{2}\\
&\bar{\Theta}^{2}=2[(C_{h}^{x}+(C_{h}^{z})^{2}+(C_{h}^{k})^{2}+2C_{h}^{\nu}+2C_{h}^{y})]\Theta^{2}+2(C_{h}^{x}+C_{h}^{\nu}),
\end{align*}
we obtain 
\begin{equation}
 \E[\int_{t}^{T}(|\Delta{Z}_{s}|^{2}+|\Delta{K}|^{2}_{s})ds\leq \bar{\Theta}^{1}\E[|\Delta{X}_{T}|^{2}]+\bar{\Theta}^{2}\int_{0}^{T}\E|\Delta{X}_{s}|^{2}ds.
 \end{equation}
 \end{proof}
 The previous estimates allow to prove the following uniqueness of the solution of the Mean field FBSDE with jumps $\textbf{(S)}$.
\begin{Proposition}\label{unique H1}
Under $\textbf{(H1)}$, there exists a unique solution $U=(X,Y,Z,K)$ of the  mean field FBSDE with jumps $\textbf{(S)}$.
\end{Proposition}
 
\begin{proof}
Suppose that $\textbf{(S)}$ has another solution $U'=(X',Y',Z',K')$. 
Applying It\^o's formula to the product $\Delta X_{t}\Delta Y_{t}$ gives
\begin{align*}
&d(\Delta X_{t}\Delta Y_{t})=\Delta X_{t}d(\Delta Y_{t})+\Delta Y_{t}d(\Delta X_{t})+d\langle \Delta X_{t},\Delta Y_{t}\rangle_{t}
\end{align*}
Taking the conditional expectation, we obtain  
\begin{align*}
\Gamma_{T}&=\mathbb{E}[ \Delta X_{T}\Delta Y_{T}]=\mathbb{E}\Big[\displaystyle  \int_{0}^{T}\{(f(s,U_{s},\nu_{s})-f(s,U'_{s},\nu'_{s}))\Delta Y_{s}\\
&+(h(s,U_{s},\nu_{s})-h(s,U'_{s},\nu'_{s})\Delta X_{s}+(\sigma(s,U_{s},\nu_{s})-\sigma(s,U'_{s},\nu'_{s}))\Delta Z_{s}\}ds\\
&+\int_{0}^{T}\int_{E}(\beta(s,U_{s},\nu_{s})-\beta(s,U'_{s},\nu'_{s}))\Delta K_{s}\eta(ds,de)\Big]+\E[\int_{0}^{T}\Delta X_s\Delta Z_sdW_{s}]\\
&+\int_{0}^{T}\int_{E}\Delta X_s\Delta K_{s}\tilde{\pi}(ds,de)]+\E[\int_{0}^{T}\Delta Y_s(\sigma(s,U_s,\nu)-\sigma(s,U^{'}_s,\nu_s^{'})°dW_{s}]\\
&+\E[\int_{0}^{T}\int_{E}\Delta Y_s(\beta(s,U_s,\nu_s)-\beta(s,U^{'}_s,\nu_s^{'}))\tilde{\pi}(de,ds)].
\end{align*}
 
\noindent Let us observe that the local martingale $\int_{0}^{t}\Delta X_s\Delta Z_sdW_{s}+\int_{0}^{T} \Delta X_s\Delta K_{s}\tilde{\pi}(ds,de)$ is a true $(\P,\mathcal{F})$ martingale. Indeed, using the BDG inequality with the help of the square integrability of $\Delta Y$, $\Delta Z$ and $\Delta K$ 
\begin{align*}
\E\Big[\sup_{0\leq t\leq T}| \int_{0}^{t}\Delta X_s\Delta Z_sdW_{s}|\Big]
&\leq C \E\big[\sup_{0\leq t\leq T}|\Delta X_{t}|^{2}\int_{0}^{T}|\Delta Z_{s}|^{2 }ds\big]^{\frac{1}{2}
}\\
&\leq C( \E[\sup_{0\leq t\leq T}|\Delta X_t|^{2}]+\E[\int_{0}^{T}|\Delta Z_{s}|^2ds])<+\infty,
\end{align*}
and 
\begin{align*}
\E[\sup_{0\leq t\leq T}|\int_{0}^{t}\int_{E}\Delta X_s\Delta K_{s}\tilde{\pi}(de,ds)]
&\leq C\E[\int_{0}^{T}\int_{E}|\Delta X_{s}\Delta K_{s}|^{2}\eta(de,ds)]^{\frac{1}{2}}\\
&\leq C(\E[\sup_{0\leq t\leq T}|\Delta X_{t}|^{2}]+\E[\int_{0}^{T}\int_{E}|\Delta K_{s}|^{2}\eta(de,ds)])<+\infty.
\end{align*}
 In the same way, we can prove that $\int_{0}^{T}\Delta Y_s(\sigma(s,U_s,\nu)-\sigma(s,U^{'}_s,\nu_s^{'})dW_{s}+\int_{0}^{T}\Delta Y_s(\beta(s,U_s,\nu_s)-\beta(s,U^{'}_s,\nu_s^{'})\tilde{\pi}(de,ds)$ is a $(\P,\mathcal{F})$-martinagle.\\
We now study each term separately. Let us start by the term  $\Delta X_{T}\Delta Y_{T}$: Using $\textbf{(H1)}$, we compute
\begin{align}\label{born-inf}
\Gamma_{T}&=\mathbb{E}[(\Delta X_{T})(g(X'_{T},\mathbb{P}_{X'_{T}})-g(X_{T},\mathbb{P}_{X_{T}}))]\nonumber\\
&\geq \mathbb{E}[k'|\Delta X_{T}|^{2}-C^{\nu}_{g}\big(|\Delta X_{T}|.\mathcal{W}_{2}(\mathbb{P}_{X'_{T}},\mathbb{P}_{X_{T}})\big)]\nonumber\\
&\geq k'\mathbb{E}[|\Delta X_{T}|^{2}]-C^{\nu}_{g}\mathbb{E}[|\Delta X_{T}|]\mathbb{E}[|\Delta X_{T}|^{2}]^{\frac{1}{2}}\nonumber\\
&\geq (k'-C^{\nu}_{g})\mathbb{E}[|\Delta X_{T}|^{2}].
\end{align}

\noindent On the other hand, we have
\begin{align*}
\Gamma_{T}\leq \mathbb{E}\Big[&\displaystyle\int_{0}^{T}\{\mathcal{A}(s,U_{s},U'_{s},\nu_{s})+((f(s,U'_{s},\nu_{s})-f(s,U'_{s},\nu'_{s})).\Delta Y_{s}\\
&+(h(s,U'_{s},\nu_{s})-h(s,U'_{s},\nu'_{s})).\Delta X_{s}+(\sigma(s,U'_{s},\nu_{s})-\sigma(s,U'_{s},\nu'_{s}))\Delta Z_{s}\}ds\nonumber\\
&+\int_{0}^{T}\int_{E}(\beta(s,U'_{s},\nu_{s})-\beta(s,U'_{s},\nu'_{s}))\Delta K_{s}\eta(ds,de)\Big].\nonumber
\end{align*}
The Lipschitz assumption and Young inequality: $ ab\leq \frac{1}{2}(a^{2}+b^{2})$ entails
\begin{align*}
\Gamma_{T}&\displaystyle{\leq\mathbb{E}\Big[\int_{0}^{T}[\mathcal{A}(s,U_{s},U^{'}_{s},\nu)+\big(C^{\nu}_{h} |\Delta X_{s}|+C^{\nu}_{f}|\Delta Y_{s}|+C^{\nu}_{\sigma}|\Delta Z_{s}|+C^{\nu}_{\beta}|\Delta K_s|_{s}\big)\mathcal{W}_{2}(\nu_{s},\nu_{s}^{'})]ds\Big]}\nonumber\\
&\leq\mathbb{E}[\displaystyle{\int_{0}^{T}-k|\Delta{X}_{s}|^{2}ds]+\frac{1}{2}
\mathbb{E}[\int_{0}^{T}(C_{h}^{\nu}|\Delta{X}_{s}|^{2}
+C_{h}^{\nu}\mathcal{W}_{2}^{2}(\nu^{'}_{s},\nu_{s})+C_{f}^{\nu}|\Delta{Y}_{s}|^{2}
+C_{f}^{\nu}\mathcal{W}_{2}^{2}(\nu^{'}_{s},\nu_{s}))ds]}\nonumber\\
&+\frac{1}{2}\mathbb{E}[\int_{0}^{T}(C_{\sigma}^{\nu}\|\Delta{Z}_{s}\|^{2}
+C_{\sigma}^{\nu}\mathcal{W}_{2}^{2}(\nu_{s}^{'},\nu_{s}))ds
+\frac{1}{2}\mathbb{E}[\int_{0}^{T}(C_{\beta}^{\nu}|\Delta K_{s}|_{s}^{2}
+C_{\beta}^{\nu}\mathcal{W}_{2}^{2}(\nu^{'}_{s},\nu_{s}))ds]\nonumber
\end{align*}
Observe that 
\begin{equation}
\mathcal{W}_{2}^{2}(\nu'_{s},\nu_{s})\leq \E[|\Delta X_{s}|^{2}]+E[|\Delta Y_{s}|^{2}],
\end{equation}
we obtain 
\begin{align*}
&\Gamma_{T}\leq\mathbb{E}[\displaystyle\int_{0}^{T}-k|\Delta X_{s}|^{2}ds\Big]+\E[\int_{0}^{T}[C_{h}^{\nu}+\frac{1}{2}(C_{f}^{\nu}+C_{\sigma}^{\nu}+C_{\beta}^{\nu})]|\Delta X_{s}|^{2}ds]\\
&+\E[\int_{0}^{T}[C_{f}^{\nu}+\frac{1}{2}(C_{h}^{\nu}+C_{\sigma}^{\nu}+C_{\beta}^{\nu})]|\Delta Y_{s}|^{2}ds]+\frac{1}{2}C_{\sigma}^{\nu}\E[\int_{0}^{T}\|\Delta Z_{s}\|^{2}ds]+\frac{1}{2}C_{\beta}^{\nu}\E[\int_{0}^{T}|\Delta K_{s}|_{s}^{2}ds].
\end{align*}
Now using the estimates in lemma \eqref{lemme estimation constante MBSDE}, we get 
\begin{align*}
&\Gamma_{T}\leq\mathbb{E}\Big[\displaystyle\int_{0}^{T}(-k+[C_{h}^{\nu}+\frac{1}{2}(C_{f}^{\nu}+C_{\sigma}^{\nu}+C_{\beta}^{\nu})]|\Delta X_{s}|^{2}ds]\\
&+[C_{f}^{\nu}+\frac{1}{2}(C_{h}^{\nu}+C_{\sigma}^{\nu}+C_{\beta}^{\nu})] \Big(\Theta^{1}\E[|\Delta{X}_{T}|^{2}]+\Theta^{2}\int_{0}^{T}\E|\Delta{X}_{s}|^{2}ds\Big)\\
&+\frac{1}{2}(C_{\sigma}^{\nu}\wedge C_{\beta}^{\nu})\big(\bar{\Theta}^{1}\E[|\Delta{X}_{T}|^{2}]+\bar{\Theta}^{2}\int_{0}^{T}\E|\Delta{X}_{s}|^{2} ds\big)\Big].\nonumber
\end{align*}
Hence  \eqref{born-inf} gives that
\begin{align}
&(k'-C^{\nu}_{g})\mathbb{E}[|\Delta X_{T}|^{2}]\leq \Big(\Theta^{1}\big[C_{f}^{\nu}+\frac{1}{2}(C_{h}^{\nu}+C_{\sigma}^{\nu}+C_{\beta}^{\nu})\big]+\frac{1}{2}(C_{\sigma}^{\nu}\wedge C_{\beta}^{\nu})\bar{\Theta}^{1}\Big)\mathbb{E}[|\Delta X_{T}|^{2}]\\
&\Big(-k+[C_{h}^{\nu}+\frac{1}{2}(C_{f}^{\nu}+C_{\sigma}^{\nu}+C_{\beta}^{\nu})]+[C_{f}^{\nu}+\frac{1}{2}(C_{h}^{\nu}+C_{\sigma}^{\nu}+C_{\beta}^{\nu})] \Theta^{2}+\frac{1}{2}(C_{\sigma}^{\nu}\wedge C_{\beta}^{\nu})\bar{\Theta}^{2}
\Big)\mathbb{E}[\int_{0}^{T}|\Delta X_{s}|^{2}ds],\nonumber
\end{align}
Henceforth taking 
$$(k'-C^{\nu}_{g})\geq\Theta^{1}\big[C_{f}^{\nu}+\frac{1}{2}(C_{h}^{\nu}+C_{\sigma}^{\nu}+C_{\beta}^{\nu})\big]+\frac{1}{2}(C_{\sigma}^{\nu}\wedge C_{\beta}^{\nu})\bar{\Theta}^{1}$$ and 
$$k \geq [C_{h}^{\nu}+\frac{1}{2}(C_{f}^{\nu}+C_{\sigma}^{\nu}+C_{\beta}^{\nu})]+[C_{f}^{\nu}+\frac{1}{2}(C_{h}^{\nu}+C_{\sigma}^{\nu}+C_{\beta}^{\nu})] \Theta^{2}+\frac{1}{2}(C_{\sigma}^{\nu}\wedge C_{\beta}^{\nu})\bar{\Theta}^{2},$$
we  obtain that $\forall t\in[0,T], X^{'}_{t}=X_{t}$ and then we obtain the uniqueness of the solution of the FBSDEJ \eqref{FBSDE}.
\end{proof}
\begin{Theorem}\label{existence H1}
Under Assumption $\textbf{(H1)}$, there exists a solution $U=(X,Y,Z,K)$ of the  mean field FBSDE with jumps $\textbf{(S)}$.
\end{Theorem}
\begin{proof}
In this part, in order to prove the existence,  we use an approximation scheme based on perturbations of the Forward SDEJs. Let $\delta \in ]0,1]$ and consider a sequence $(X^{n},Y^{n},Z^{n},K^{n})$ of processes defined recursively by :
$(X^{0},Y^{0},Z^{0},K^{0})=(0,0,0,0)$ and for $n\geq 1$, $U^{n}=(X^{n},Y^{n},Z^{n},K^{n})$ satisfies 
\begin{equation}\label{system 2}
{\small{
\begin{cases}
\vspace{0.2cm}
&\small{\displaystyle{X^{n+1}_{t}}=\displaystyle{X_{0}+\int_{0}^{t}[b_s(U^{n+1}_{s},\nu^{n}_{s})-\delta(Y^{n+1}_{s}-Y^{n}_{s})]ds+\int_{0}^{t}[\sigma_s(U^{n+1}_{s},\nu^{n}_{s})-\delta(Y^{n+1}_{s}-Y^{n}_{s})]dW_{s}}}\\\vspace{0.3cm}
& \qquad \hspace{0.3cm}\displaystyle{+\int_{0}^{t}\int_{E}\Big(\beta_s(U^{n+1}_{s},\nu^{n}_{s})-\delta(K^{n+1}_{s}-K^{n}_{s})\Big)\tilde{\pi}(ds,de),}\\
&\displaystyle{Y_{t}^{n+1}=g(X^{n+1}_{T},\mu^{n}_{T})-\int_{t}^{T}h_s(U^{n+1}_{s},\nu^{n}_{s})ds-\int_{t}^{T}Z^{n+1}_{s}dW_{s}-\int_{t}^{T}\int_{E}K^{n+1}_{s}(e)\tilde{\pi}(ds,de).}
\end{cases}
}}
\end{equation}
Hereafter, we shall use the following simplified notations: For $n\geq 1$, $t \in[0,T]$, we set
$$
\hat{X}^{n+1}_{t}:=X^{n+1}_{t}-X^{n}_{t},\quad \hat{Y}^{n+1}_{t}:=Y^{n+1}_{t}-Y^{n}_{t},\quad \hat{Z}^{n+1}_{t}:=Z^{n+1}_{t}-Z^{n}_{t},\quad \hat{K}^{n+1}_{t}:=K^{n+1}_{t}-K^{n}_{t}
$$ 
and for a function $\phi=\{b,h,\sigma,\beta\}$, we set 
$$
\hat{\phi}^{n+1}_{t}:=\phi(t,U^{n+1}_{t},\nu^{n}_{t})-\phi(t,U^{n}_{t},\nu^{n-1}_{t}), \quad \tilde{\phi}^{n}_{t}:=\phi(t,U^{n}_{t},\nu^{n}_{t})-\phi(t,U^{n}_{t},\nu^{n-1}_{t}).
$$
We first apply It\^o's formula to the product $\hat{ X}^{n+1}\hat{Y}^{n+1}$
\begin{align*}
\E[&\displaystyle\hat{X}_{T}^{n+1}\hat{Y}_{T}^{n+1}]=\displaystyle{\E[\int_{0}^{T}\hat{Y}^{n+1}_{s}[\hat{b}^{n+1}_{s}-\delta (\hat{Y}^{n+1}_{s}-\hat{Y}^{n}_{s})]ds+\E[\int_{0}^{T}\hat{Y}^{n+1}_{s}[\hat{\sigma}^{n+1}_{s}-\delta(\hat{Z}^{n+1}_{s}-\hat{Z}^{n}_{s})]dW_{s}]}\\
&+\displaystyle{\E[\int_{0}^{T}\int_{E}\hat{Y}^{n+1}_{s}[\hat{\beta}^{n+1}_{s}-\delta(\hat{K}^{n+1}_{s}(e)-\hat{K}^{n}_{s}(e)]\tilde{\pi}(de,ds)
+\E[\int_{0}^{T}\hat{X}^{n+1}_{s}\hat{h}^{n+1}_{s}ds]}\\
&-\displaystyle{\E[\int_{0}^{T}\hat{X}^{n+1}_{s}\hat{Z}^{n+1}_{s}dW_{s}]
-\E[\int_{0}^{T}\int_{E}\hat{X}^{n+1}_{s}\hat{K}^{n+1}_{s}(e)\tilde{\pi}(de,ds)]}\\
&+\displaystyle{\E[\int_{0}^{T}(\hat{\sigma}^{n+1}_{s}-\delta(\hat{Z}^{n+1}_{s}-\hat{Z}^{n}_{s}),\hat{Z}^{n+1}_{s})ds]+\E[\int_{0}^{T}\int_{E}\hat{K}^{n+1}_{s}(\hat{\beta}^{n+1}_{s}-\delta(\hat{K}^{n+1}_{s}-\hat{K}^{n}_{s}))\eta(de,ds)].}
\end{align*}
Using the BDG inequality, we can easily see that the stochastic integrals in the above expression is a true martingale. Hence  we obtain 
\begin{align*}
&\E[\hat{X}_{T}^{n+1}\hat{Y}_{T}^{n+1}]=\E[\int_{0}^{T}\hat{Y}^{n+1}_{s}[\hat{b}^{n+1}_{s}-\delta (\hat{Y}^{n+1}_{s}-\hat{Y}^{n}_{s})]ds+\E[\int_{0}^{T}\hat{X}^{n+1}_{s}\hat{h}^{n+1}_{s}ds]\nonumber\\
&+\E[\int_{0}^{T}(\hat{\sigma}^{n+1}_{s}-\delta(\hat{Z}^{n+1}_{s}-\hat{Z}^{n}_{s}),\hat{Z}^{n+1}_{s})ds]+\E[\int_{0}^{T}\int_{E}\hat{K}^{n+1}_{s}(\hat{\beta}^{n+1}_{s}-\delta(\hat{K}^{n+1}_{s}-\hat{K}^{n}_{s}))\eta(de,ds)].
\end{align*}
Rearranging terms, we get 
\begin{align}\label{Ito}
&\delta \E[\int_{0}^{T}\hat{Y}^{n+1}_{s}\hat{Y}^{n}_{s}ds+\int_{0}^{T}\hat{Z}^{n+1}_{s}\hat{Z}^{n}_{s}ds+\int_{0}^{T}\int_{E}\hat{K}^{n+1}_{s}\hat{K}^{n}_{s}\eta(de,ds)]=\E[\hat{X}_{T}^{n+1}\hat{Y}_{T}^{n+1}]\nonumber\\
&-\E[\int_{0}^{T}\hat{X}^{n+1}_{s}\hat{h}^{n+1}_{s}+\hat{Y}^{n+1}_{s}\hat{b}^{n+1}_{s}+\hat{Z}^{n+1}_{s}\hat{\sigma}^{n+1}_{s}]ds+\int_{0}^{T}\int_{E}\hat{K}^{n+1}_{s}\hat{\beta}^{n+1}_{s}\eta(de,ds)].\nonumber\\
&+\delta\E[\int_{0}^{T}|\hat{Y}^{n+1}_{s}|^{2}+\|\hat{Z}^{n+1}_{s}\|^{2}+|\hat{K}^{n+1}_{s}|_{s}^{2}ds].
\end{align}
Since $\hat{X}^{n+1}_{T}\hat{Y}^{n+1}_{T}=\hat{X}^{n+1}_{T}[g(X^{n+1}_{T},\mu^{n}_{T})-g(X^{n}_{T},\mu^{n-1}_{T})]$, we have from  $\textbf{(H1)}$
\begin{align*}
\hat{X}_{T}^{n+1}\hat{Y}_{T}^{n+1}&=\hat{X}^{n+1}_{T}[g(X_{T}^{n+1},\mu^{n}_{T})-g(X_{T}^{n},\mu^{n-1}_{T})]\\
&=\hat{X}^{n+1}_{T}[g(X_{T}^{n+1},\mu^{n}_{T})-g(X_{T}^{n},\mu^{n}_{T})]+\hat{X}^{n+1}_{T}[g(X_{T}^{n},\mu^{n}_{T})-g(X_{T}^{n},\mu^{n-1}_{T})]\\
&\geq k^{'}|\hat{X}^{n}_{T}|^{2}-C_{g}^{\nu}|\hat{X}^{n}_{T}|\mathcal{W}_{2}(\mu^{n}_{T},\mu^{n-1}_{T}).
\end{align*}
Using \eqref{Wasserstein} and the elementary inequality : $\forall \epsilon >0,\,  2ab\leq \epsilon^{-1}a^{2}+\epsilon b^{2}$, we obtain 
\begin{equation}\label{estimates I}
\E[\hat{X}_{T}^{n+1}\hat{Y}_{T}^{n+1}]\geq (k^{'}-\frac{C_{g}^{\nu}\epsilon}{2})\E[|\hat{X}^{n+1}_{T}|^{2}]-\frac{C_{g}^{\nu}}{2\epsilon}\E[|\hat{X}^{n}_{T}|^{2}].
\end{equation}
In the other hand, using once again $\textbf{(H1)}$
\begin{align*}
\displaystyle&\int_{0}^{T}[\hat{X}^{n+1}_{s}\hat{h}^{n+1}_{s}+\hat{Y}^{n+1}_{s}\hat{b}^{n+1}_{s}+\hat{Z}^{n+1}_{s}\hat{\sigma}^{n+1}_{s}]ds+\int_{0}^{T}\int_{E}\hat{K}^{n+1}_{s}\hat{\beta}^{n+1}_{s}\eta(de,ds)\\
\vspace{0.2cm}
=\displaystyle&\,\int_{0}^{T}[\mathcal{A}(s,U_{s}^{n+1},U_{s}^{n},\nu^{n})+\hat{Y}^{n+1}_{s}\bar{b}^{n}_{s}+\hat{X}^{n+1}_{s}\bar{h}^{n}_{s}+\hat{Z}^{n+1}_{s}\bar{\sigma}^{n}_{s}]ds+\int_{0}^{T}\int_{E}\hat{K}^{n+1}_{s}\bar{\beta}^{n}_{s}\eta(de,ds).\\
\displaystyle\leq &-k\int_{0}^{T}| \hat{X}^{n+1}_{s}|^{2}ds+\int_{0}^{T}[C_{h}^{\nu}| \hat{X}^{n+1}_{s}|+C_{f}^{\nu}| \hat{Y}^{n+1}_{s}|+C_{\sigma}^{\nu}\| \hat{Z}^{n+1}_{s}\|+C_{\beta}^{\nu}| \hat{K}^{n+1}_{s}|_{s}]\mathcal{W}_{2}(\nu^{n},\nu^{n-1})ds.
\end{align*}
Using Young inequality : $\forall \tilde{\epsilon} >0,\,  2ab\leq \tilde{\epsilon}^{-1}a^{2}+\tilde{\epsilon} b^{2}$, we obtain 
\begin{align*}
&\int_{0}^{T}[\hat{X}^{n+1}_{s}\hat{h}^{n+1}_{s}+\hat{Y}^{n+1}_{s}\hat{b}^{n+1}_{s}+\hat{Z}^{n+1}_{s}\hat{\sigma}^{n+1}_{s}]ds+\int_{0}^{T}\int_{E}\hat{K}^{n+1}_{s}\hat{\beta}^{n+1}_{s}\eta(de,ds)\\
&\leq (\frac{ \tilde{\epsilon}C_{h}^{\nu}}{2}-k)\int_{0}^{T}| \hat{X}^{n+1}_{s}|^{2}ds+\frac{\tilde{\epsilon }}{2}\int_{0}^{T}(C_{f}^{\nu}| \hat{Y}^{n+1}_{s}|^{2}+C_{\sigma}^{\nu}\| \hat{Z}^{n+1}_{s}\|^{2}+C_{\beta}^{\nu}|\hat{K}^{n+1}_{s}|_{s}^{2})ds\\
&+\frac{C_{h}^{\nu}+C_{f}^{\nu}+C_{\sigma}^{\nu}+C_{\beta}^{\nu}}{2\tilde{\epsilon}}\mathcal{W}_{2}^{2}(\nu_{s}^{n},\nu_{s}^{n-1}).
\end{align*}
Notice that  $\mathcal{W}_{2}^{2}(\nu^{n}_{s},\nu^{n-1}_{s})\leq\E[|\hat{X}^{n}_{s}|^{2}+|\hat{Y}^{n}_{s}|^{2}]$. Hence taking the expectation we obtain 
\begin{align}\label{estimatesI-1}
&\E[\int_{0}^{T}[\hat{X}^{n+1}_{s}\hat{h}^{n+1}(s)+\hat{Y}^{n+1}_{s}\hat{b}^{n+1}(s)+\hat{Z}^{n+1}_{s}\hat{\sigma}^{n+1}(s)]ds+\int_{0}^{T}\int_{E}\hat{K}^{n+1}_{s}\hat{\beta}^{n+1}_{s}\eta(de,ds)]\nonumber\\
&\leq (\frac{ \tilde{\epsilon}C_{h}^{\nu}}{2}-k)\E[\int_{0}^{T}| \hat{X}^{n+1}_{s}|^{2}ds]+\frac{\tilde{\epsilon }}{2}\E[\int_{0}^{T}(C_{f}^{\nu}| \hat{Y}^{n+1}_{s}|^{2}+C_{\sigma}^{\nu}\| \hat{Z}^{n+1}_{s}\|^{2}+C_{\beta}^{\nu}|\hat{K}^{n+1}_{s}|_{s}^{2})ds]\nonumber\\
&+\frac{C_{h}^{\nu}+C_{f}^{\nu}+C_{\sigma}^{\nu}+C_{\beta}^{\nu}}{2\tilde{\epsilon}}\E[\int_{0}^{T}(|\hat{X}^{n}_{s}|^{2}+|\hat{Y}^{n}_{s}|^{2})ds].
\end{align}
In addition,
 \begin{align}\label{estimatesI-2}
&\E[\int_{0}^{T}(\hat{Y}^{n+1}_{s}\hat{Y}^{n}_{s}+\hat{Z}^{n+1}_{s}\hat{Z}^{n}_{s})ds+\int_{0}^{T}\int_{E}\hat{K}^{n+1}_{s}\hat{K}^{n}_{s}\,\eta(de,ds)]\nonumber\\
\leq \frac{\kappa}{2}&\E[\int_{0}^{T} |\hat{Y}^{n+1}_{s}|^{2}+\|\hat{Z}^{n+1}_{s}\|^{2}+| \hat{K}^{n+1}_{s}|_{s}^{2}ds]+\frac{1}{2\kappa}\E[\int_{0}^{T}|\hat{Y}^{n}_{s}|^{2}+\|\hat{Z}^{n}_{s}\|^{2}+| \hat{K}^{n}_{s}|_{s}^{2}ds].
\end{align}
Plugging \eqref{estimates I}, \eqref{estimatesI-1} and \eqref{estimatesI-2} in \eqref{Ito} we obtain 
\begin{align*}
&(k^{'}-\frac{C^{\nu}_{g}\epsilon}{2 })\E[|\hat{X}^{n+1}_{T}|^{2}]-\frac{C^{\nu}_{g}}{2 \epsilon}\E[|\hat{X}^{n}_{T}|^{2}]+\delta \E[\int_{0}^{T}|\hat{Y}^{n+1}_{s}|^{2}+\|\hat{Z}^{n+1}_{s}\|^{2}+|\hat{K}^{n+1}_{s}|_{s}^{2}ds]\\
&+(-\frac{\tilde{\epsilon}C^{\nu}_{h}}{2}+k) \E[\int_{0}^{T}|\hat{X}^{n+1}_{s}|^{2}ds]-\frac{\tilde{\epsilon}}{2}\E[\int_{0}^{T}C^{\nu}_{f}|\hat{Y}^{n+1}_{s}|^{2}+C^{\nu}_{\sigma}\|\hat{Z}^{n+1}_{s}\|^{2}+C^{\nu}_{\beta}|\hat{K}^{n+1}_{s}|_{s}^{2}ds]\\
&-\frac{C_{h}^{\nu}+C_{f}^{\nu}+C_{\sigma}^{\nu}+C_{\beta}^{\nu}}{2\epsilon}\E[\int_{0}^{T}|\hat{X}^{n}_{s}|^{2}+|\hat{Y}^{n}_{s}|^{2}ds]\\
&
\leq\frac{\delta \kappa}{2} \E[\int_{0}^{T}|\hat{Y}^{n+1}_{s}|^{2}+\|\hat{Z}^{n+1}_{s}\|^{2}+|\hat{K}^{n+1}_{s}|_{s}^{2}ds]+ \frac{ \delta}{2\kappa}\E[\int_{0}^{T}|\hat{Y}^{n+1}_{s}|^{2}+\|\hat{Z}^{n+1}_{s}\|^{2}+|\hat{K}^{n+1}_{s}|_{s}^{2}ds]
.
\end{align*}
Rearranging terms we get
\begin{align*}
\hspace{0.1cm}
(k^{'}-\frac{C^{\nu}_{g}\epsilon}{2 })&\E[|\hat{X}^{n+1}_{T}|^{2}]+(k-\frac{\tilde{\epsilon}C^{\nu}_{h}}{2}) \E[\int_{0}^{T}|\hat{X}^{n+1}_{s}|^{2}ds]
+(\delta-\frac{\kappa\delta}{2}-\frac{\tilde{\epsilon}C^{\nu}_{f}}{2})\E[\int_{0}^{T}|\hat{Y}^{n+1}_{s}|^{2}ds]\\
&+(\delta-\frac{\kappa\delta}{2}-\frac{\tilde{\epsilon}C^{\nu}_{\sigma}}{2})\E[\int_{0}^{T}\|\hat{Z}^{n+1}_{s}\|^{2}ds]+(\delta-\frac{\kappa\delta}{2}-\frac{\tilde{\epsilon}C^{\nu}_{\beta}}{2})\E[\int_{0}^{T}|\hat{K}^{n+1}_{s}|_{s}^{2}ds]\\
&\leq\frac{C^{\nu}_{g}}{2 \epsilon}\E[|\hat{X}^{n}_{T}|^{2}]+\frac{C_{h}^{\nu}+C_{f}^{\nu}+C_{\sigma}^{\nu}+C_{\beta}^{\nu}}{2\epsilon}\E[\int_{0}^{T}|\hat{X}^{n}_{s}|^{2}ds]\\
&+\frac{\delta}{2\kappa}E[\int_{0}^{T}\|\hat{Z}^{n}_{s}\|^{2}+|\hat{K}^{n}_{s}|_{s}^{2}ds]+(\frac{C_{h}^{\nu}+C_{f}^{\nu}+C_{\sigma}^{\nu}+C_{\beta}^{\nu}}{2\epsilon}+\frac{\delta}{2\kappa})\E[\int_{0}^{T}|\hat{Y}^{n}_{s}|^{2}ds].
\end{align*}
Setting 
\begin{equation}
 \begin{cases}
   & \gamma :=min (k^{'}-\frac{C^{\nu}_{g}\epsilon}{2 },k-\frac{\tilde{\epsilon}C^{\nu}_{h}}{2},(\delta-\frac{\kappa\delta}{2}-\frac{\tilde{\epsilon}C^{\nu}_{f}}{2}),(\delta-\frac{\kappa\delta}{2}-\frac{\tilde{\epsilon}C^{\nu}_{\sigma}}{2}),(\delta-\frac{\kappa\delta}{2}-\frac{\tilde{\epsilon}C^{\nu}_{\beta}}{2}))\\
   &\theta=max (\frac{C^{\nu}_{g}}{2 \epsilon},-\frac{C_{h}^{\nu}+C_{f}^{\nu}+C_{\sigma}^{\nu}+C_{\beta}^{\nu}}{2\epsilon}+\frac{\delta}{2\kappa}),
    \end{cases}
\end{equation}
we obtain that 
\begin{equation}
\E[|\hat{X}^{n+1}_{T}|^{2}]+\E[\int_{0}^{T}\|\hat{U}^{n+1}_{s}\|^{2}_{s}ds]\leq \frac{\theta}{\gamma}\E[|\hat{X}^{n}_{T}|^{2}]+\E[\int_{0}^{T}|\hat{U}^{n}_{s}|^{2}_{s}ds].
\end{equation}
Choosing $\tilde{\epsilon}$ and $\epsilon$ so that $\theta <\gamma$, the inequality becomes a contraction. Thus, $(\hat{X}^{n}_{T})_{n\geq 0}$ is a Cauchy sequence in $\mathcal{L}^{2}(\Omega,\P)$ and $(\hat{X}^{n})_{n\geq 0},(\hat{Y}^{n})_{n\geq 0}, (\hat{Z}^{n})_{n\geq 0}$ and $\hat{K}^{n})_{n\geq 0})$ are Cauchy sequences respectively in $\mathcal{L}^{2}([0,T],\Omega,dt\otimes d\P)$and $\mathcal{L}^{2}_{\nu}([0,T],\Omega,dt\otimes d\nu).$ Hence, if $X, Y$, $Z$ and $K$ are the respective limits of these sequences, passing to the limit in \eqref{system 2}, we see that $(X,Y,Z,K)$ is a solution of \eqref{FBSDE}.
\end{proof}
\subsection{Existence and uniqueness under $\textbf{(H2)}$}
Our second main result is the extension to the case where the datas satisfy a weaker monotonicity assumptions. We adopt here a common strategy which is the Picard approach: we construct a schema based on small perturbation. This helps us to construct the contracting maps and 
therefore deduce the existence of a unique solution of the system $\textbf{(S)}$.

Consider the following assumption 
\begin{equation*}
\textbf{(H2)}\,\,\,
\label{h2-assump}
\begin{cases} 
(i)\mbox{ There exists } k>0, \mbox{ s.t } \forall t \in[0,T], \nu\in \mathbb{M}_{1}(\mathbb{R}^{d}\times \mathbb{R}^{d}), u, u' \in \mathbb{R}^{d+d+d \times d}, \,\,\,\,\\
 \qquad \qquad \mathcal{A}(t,u,u',\nu)\leq -k(|x-x|^{2}+|y-y'|^{2}+||z-z'||^{2}+|k-k^{'}|_{t}), \mathbb{P}\mbox{-a.s.}\\
(ii)\mbox{ There exists } k'>0, \mbox{ s.t } \forall \nu\in \mathbb{M}_{1}(\mathbb{R}^{m}\times \mathbb{R}^{d}),x,x'\in \mathbb{R}^{m}\\
\qquad \qquad(g(x,\nu)-g(x',\nu)).(x-x')\leq  k'|x-x'|^{2}, \mathbb{P}\mbox{-a.s.}
\end{cases}
\end{equation*}
As for the previous section we will give some useful a priori estimates.
\begin{Lemma}\label{lemme estimation constante MBSDE}
Let $(Y^{'},Z^{'},K^{'})$ another solution  of the the system $\textbf{(S)}$. Then, under $(\textbf{H2})$ we have the following estimates 
 \begin{align*}
 & \E[\int_{0} ^{T}| \Delta{X}_{s}|^{2}ds] \leq \frac{\exp(t.\Upsilon^{1})-\Upsilon^{1}}{\Upsilon^{1}}[\Upsilon^{2}E[\int_{0}^{t}|\Delta{Y}_{s}|^{2}ds
  +\Upsilon^{3}E[\int_{0}^{t}|\Delta{Z}_{s}|^{2}ds+\Upsilon^{4}E[\int_{0}^{t}|\Delta{K}_{s}|_{s}^{2}ds],
   \end{align*}
   where
   \begin{equation}\label{constante H2}
\begin{cases}
   &\Upsilon^{1}:= 2C_{f}^{x}+C_{f}^{y}+C_{f}^{z}+C_{f}^{k}+5(C_{\sigma}^{x})^{2}+5(C_{\beta}^{x})^{2}+C_{f}^{\nu}+C_{f}^{\nu}+5(C_{\sigma}^{\nu})^{2}+5(C_{\beta}^{\nu})^{2} \nonumber \\
   &\Upsilon^{2}:=C_{f}^{y}+5(C_{\sigma}^{y})^{2}+5(C_{ \beta}^{y})^{2}+C_{f}^{\nu}+5(C_{\sigma}^{\nu})^{2}+5(C_{\beta}^{\nu})^{2}\nonumber\\
   &\Upsilon^{3}:=(C_{f}^{z}+5(C_{\sigma}^{z})^{2}+5(C_{\beta}^{x})^{2}\nonumber\\
   &\Upsilon^{4}:=(C_{f}^{k}+5(C_{\sigma}^{k})^{2}+5(C_{\beta}^{k})^{2},
   \end{cases} 
   \end{equation}
 \end{Lemma}
 \begin{proof}
Applying It\^o formula to $|\Delta{X}|^{2}$, we compute using the Lipschitz assumption
 \begin{align*}
& \E[|\Delta{X}_{t}|^{2}] \leq 2\E[\int_{0}^{t}|\Delta{X}_{s}|(C^{x}_{f}|\Delta{X_{s}}|+C_{f}^{y}| \Delta{Y}_{s}|+C_{f}^{z} \|\Delta{Z}_{s}\|+C^{k}_{f}|\Delta{K}_{s}|_{s}+C_{f}^{\nu}\mathcal{W}_{2}(\nu'_{s},\nu_{s}))ds\\
&+5\E[\int_{0}^{t}[(C^{x}_{\sigma})^{2}|\Delta{X_{s}}|^{2}+(C_{\sigma}^{y})^{2}| \Delta{Y}_{s}|^{2}+(C_{\sigma}^{z})^{2} \|\Delta{Z}_{s}\|^{2}+(C^{k}_{\sigma})^{2}|\Delta{K}_{s}|^{2}_{s}+(C_{\sigma}^{\nu})^{2}\mathcal{W}_{2}^{2}(\nu'_{s},\nu_{s})]ds\\
&+5\E[\int_{0}^{t}\left[(C^{x}_{\beta})^{2}|\Delta{X_{s}}|^{2}+(C_{\beta}^{y})^{2}| \Delta{Y}_{s}|^{2}+(C_{\beta}^{z})^{2} \|\Delta{Z}_{s}\|^{2}+(C^{k}_{\beta})^{2}|\Delta{K}_{s}|^{2}_{s}+(C_{\beta}^{\nu})^{2}\mathcal{W}_{2}(\nu'_{s},\nu_{s})^{2}\right]
  \end{align*}
  Then from Young inequality we get 
   \begin{align*}
 & \E[| \Delta{X}_{t}|^{2}] \leq\Big( 2C_{f}^{x}+C_{f}^{y}+C_{f}^{z}+C_{f}^{k}+5(C_{\sigma}^{x})^{2}+5(C_{\beta}^{x})^{2}+C_{f}^{\nu}+C_{f}^{\nu}+5(C_{\sigma}^{\nu})^{2}+5(C_{\beta}^{\nu})^{2}\Big) \E[\int_{0}^{t}|\Delta{X}_{s}|^{2}ds\\
   &+(C_{f}^{y}+5(C_{\sigma}^{y})^{2}+5(C_{ \beta}^{y})^{2}+C_{f}^{\nu}+5(C_{\sigma}^{\nu})^{2}+5(C_{\beta}^{\nu})^{2}E[\int_{0}^{t}|\Delta{Y}_{s}|^{2}ds\\
   &+(C_{f}^{z}+5(C_{\sigma}^{z})^{2}+5(C_{\beta}^{x})^{2})E[\int_{0}^{t}|\Delta{Z}_{s}|^{2}ds+[C_{f}^{k}+5(C_{\sigma}^{k})^{2}+5(C_{\beta}^{k})^{2}]E[\int_{0}^{t}|\Delta{K}_{s}|_{s}^{2}ds.
   \end{align*}
   Thus, taking $\Upsilon^{1}$,  $\Upsilon^{2}$, $\Upsilon^{3}$, and $\Upsilon^{4}$ as in \eqref{constante H2} and applying
Gronwall lemma, we obtain
  \begin{equation*}
 \E[| \Delta{X}_{t}|^{2} \leq \frac{\exp(t.\Upsilon^{1})-\Upsilon^{1}}{\Upsilon^{1}}[\Upsilon^{2}E[\int_{0}^{t}|\Delta{Y}_{s}|^{2}ds
  +\Upsilon^{3}E[\int_{0}^{t}|\Delta{Z}_{s}|^{2}ds+\Upsilon^{4}E[\int_{0}^{t}|\Delta{K}_{s}|_{s}^{2}ds]],
   \end{equation*}
which gives the desired result.
 \end{proof}
\begin{Proposition}
Under assumption (H2), there exists a unique solution $(X,Y,Z,K)$  of the FBSDE with jumps \eqref{FBSDE}.
\end{Proposition}
\begin{proof}
Let $U=(X,Y,Z,K)$ and $U'=(X',Y',Z',K')$ be two solutions of the mean-field FBSDEJs $\textbf{(S)}$. Using the same notation as in Proposition \eqref{unique H1}, We have as proved earlier in \eqref{born-inf}
\begin{equation}
\label{borneinfeH2}
\Gamma_{T}\geq (k'-C_{g}^{\nu})\mathbb{E}[|\Delta X_{T}|^{2}].
\end{equation}
On the other hand, using $\textbf{(H2)}$ and the Lipschitz assumption, we compute
\begin{align*}
&\Gamma_{T}\leq\mathbb{E}[\displaystyle-k\int_{0}^{T}(|\Delta Y_{s}|^{2}+|\Delta Z_{s}|^{2}+|\Delta K_{s}|^{2})ds+\int_{0}^{T}[C_{h}^{\nu}+\frac{1}{2}(C_{f}^{\nu}+C_{\sigma}^{\nu}+C_{\beta}^{\nu})]|\Delta X_{s}|^{2}ds]\\
&+\E[\int_{0}^{T}[C_{f}^{\nu}+\frac{1}{2}(C_{h}^{\nu}+C_{\sigma}^{\nu}+C_{\beta}^{\nu})]|\Delta Y_{s}|^{2}ds+C_{\sigma}^{\nu}\int_{0}^{T}\|\Delta Z_{s}\|^{2}ds+C_{\beta}^{\nu}\int_{0}^{T}|\Delta K_{s}|_{s}^{2}ds].
\end{align*}
Combining \eqref{lemme estimation constante MBSDE} and   \eqref{borneinfeH2} we obtain
\begin{align*}
&(k'-C_{g}^{\nu})\E[|\Delta{X}_{T}|^{2}]+k\E[\int_{0}^{T}|\Delta Y_{s}|^{2}+|\Delta Z_{s}|^{2}+|\Delta K_{s}|_{s}^{2}ds]\\
& \leq\big[\Upsilon^{\nu}\frac{\exp(t.\Upsilon^{1})-\Upsilon^{1}}{\Upsilon^{1}}\Upsilon^{2}+(C_{f}^{\nu}+\frac{1}{2}(C_{h}^{\nu}+C_{\sigma}^{\nu}+C_{\beta}^{\nu})\big]E[\int_{0}^{t}|\Delta{Y}_{s}|^{2}]ds
  \\
  &+\big[\Upsilon^{\nu}\frac{\exp(t.\Upsilon^{1})-\Upsilon^{1}}{\Upsilon^{1}}\Upsilon^{3}+\frac{C^{\nu}_{\sigma}}{2}\big]E[\int_{0}^{t}|\Delta{Z}_{s}|^{2}ds]\\
&  +\big[\Upsilon^{\nu}\frac{\exp(t.\Upsilon^{1})-\Upsilon^{1}}{\Upsilon^{1}}\Upsilon^{4}+\frac{C^{\nu}_{\beta}}{2})\big]E[\int_{0}^{t}|\Delta{K}_{s}|_{s}^{2}ds].
\end{align*}
where $\Upsilon^{\nu}:=[C_{h}^{\nu}+\frac{1}{2}(C_{f}^{\nu}+C_{\sigma}^{\nu}+C_{\beta}^{\nu})]$.
Taking the Lipschitz constants small enough to obtain
\begin{align*}
&k> \big[\Upsilon^{\nu}\frac{\exp(t.\Upsilon^{1})-\Upsilon^{1}}{\Upsilon^{1}}\Upsilon^{2}+(C_{f}^{\nu}+\frac{1}{2}(C_{h}^{\nu}+C_{\sigma}^{\nu}+C_{\beta}^{\nu})\big]\\
&k>\big[\Upsilon^{\nu}\frac{\exp(t.\Upsilon^{1})-\Upsilon^{1}}{\Upsilon^{1}}\Upsilon^{3}+\frac{C^{\nu}_{\sigma}}{2}\big]\\
&k>\big[\Upsilon^{\nu}\frac{\exp(t.\Upsilon^{1})-\Upsilon^{1}}{\Upsilon^{1}}\Upsilon^{4}+\frac{C^{\nu}_{\beta}}{2})\big]
\end{align*} 
and $k'-C_{g}^{\nu}>0$. 
Thus we have $$(k'-C_{g}^{\nu})\E[|\Delta{X}_{T}|^{2}]+k\int_{0}^{T}(|Y_{s}'-Y_{s}|^{2}+\|Z_{s}'-Z_{s}\|^{2}+|K_{s}'-K_{s}|_{s}^{2})ds]\leq 0$$
This implies that $X'_{T}=X_{T}$ and for all $t\in[0,T]$, $X^{'}=X$, $Y'=Y$, $Z'=Z$ and $K'=K$  which gives the desired result.
\end{proof}
\begin{Theorem}
Under assumption $\textbf{(H2)}$, there exists a solution $(X,Y,Z,K)$  of the FBSDE with jumps \eqref{FBSDE}.
\end{Theorem}
\begin{proof}
Following the same approach as in proposition \eqref{existence H1}, we use an approximation scheme based on perturbation.   However, perturbations here are made in the backward SDE with jumps. Let $\delta >0$ and consider a sequence $(X^{n},Y^{n},Z^{n},K^{n})$ of processes defined recursively by :
$(X^{0},Y^{0},Z^{0},K^{0})=(0,0,0,0)$ and for $n\geq 1$, $U^{n}=(X^{n},Y^{n},Z^{n},K^{n})$ satisfies 
\begin{equation}\label{system H2}
\begin{cases}
\vspace{0.2cm}
&X^{n+1}_{t}=X_{0}+\int_{0}^{t}b_s(U^{n+1}_{s},\nu^{n}_{s})ds+\int_{0}^{t}\sigma_s(U^{n+1}_{s},\nu^{n}_{s})dW_{s}+\int_{0}^{t}\int_{E}\beta_s(U^{n+1}_{s},\nu^{n}_{s})\tilde{\pi}(ds,de),\\
\vspace{0.1cm}
&Y_{t}^{n+1}=g(X^{n+1}_{T},\mu^{n}_{T})+\delta(X^{n+1}_{T}-X^{n}_{T})-\int_{t}^{T}\big[h_s(U^{n+1}_{s},\nu^{n}_{s})+\delta(X^{n+1}_{s}-X^{n}_{s})\big]ds\\
&\qquad\,\,\,-\int_{t}^{T}Z^{n+1}_{s}dW_{s}-\int_{t}^{T}\int_{E}K^{n+1}_{s}(e)\tilde{\pi}(ds,de).
\end{cases}
\end{equation}
with $\mu^{n}_{T}=\P_{X^{n}_{T}},\,\, \nu^{n}_{t}=\P_{(X^{n}_{t},Y^{n}_{t})}.$\\
We keep the same notation as in Theorem \eqref{existence H1}
and we apply It\^o formula to the product  $\hat{X}_{s}^{n+1}\hat{Y}_{s}^{n+1}$
\begin{align}
&\displaystyle\hat{X}_{T}^{n+1}\hat{Y}_{T}^{n+1}-\hat{X}_{0}^{n+1}\hat{Y}_{0}^{n+1}=\int_{0}^{T}\hat{Y}^{n+1}_{s}\hat{b}^{n+1}_{s}ds+\int_{0}^{T}\hat{Y}^{n+1}_{s}\hat{\sigma}^{n+1}_{s}dW_{s}\\
&\displaystyle+\int_{0}^{T}\hat{X}^{n+1}_{s}[\hat{h}^{n+1}_{s}-\delta(\hat{X}^{n+1}_{s}-\hat{X}^{n}_{s})]ds-\int_{0}^{T}\hat{X}^{n+1}_{s}\hat{Z}^{n+1}_{s}dW_{s}-\int_{0}^{T}\hat{X}^{n+1}_{s}\hat{K}^{n+1}_{s}(e)\tilde{\pi}(de,ds)\nonumber\\
&+\displaystyle\int_{0}^{T}(\hat{\sigma}^{n+1}_{s},\hat{Z}^{n+1}_{s})ds+\int_{0}^{T}\int_{E}\hat{K}^{n+1}_{s}\hat{\beta}^{n+1}_{s}\eta(de,ds)+\int_{0}^{T}\int_{E}\hat{Y}^{n+1}_{s}\hat{\beta}^{n+1}_{s}\tilde{\pi}(de,ds).\nonumber
\end{align}
Notice that since the terminal condition is given by $\hat{Y}^{n+1}_{T}=[g(X^{n+1},\mu^{n}_{T})-g(X^{n},\mu^{n-1}_{T})]+\delta(X^{n+1}_{T}-X^{n}_{T})$ we rewrite the above equation as follows
\begin{align*}
&\E\big[\hat{X}_{s}^{n+1}(g(X_{T}^{n+1},\mu^{n}_{T})-g(X_{T}^{n},\mu^{n-1}_{T}))+\delta[|\hat{X}^{n+1}_{T}|^{2}-\hat{X}^{n+1}_{T}\hat{X}^{n}_{T}
]\big]\\
&=\E\big[\int_{0}^{T}\hat{Y}^{n+1}_{s}\hat{b}^{n+1}(s)ds+\int_{0}^{T}\hat{X}^{n+1}_{s}\hat{h}^{n+1}_{s}ds+\int_{0}^{T}(\hat{\sigma}^{n+1}(s),\hat{Z}^{n+1}_{s})ds\big]\\
&+\E\big[\int_{0}^{T}\int_{E}\hat{K}^{n+1}_{s}\hat{\beta}^{n+1}_{s}\eta(de,ds)-\delta(\int_{0}^{T}|\hat{X}^{n+1}_{s}|^{2}ds-\int_{0}^{T}\hat{X}^{n+1}_{s}\hat{X}^{n}_{s}ds)\big].
\end{align*}
Using Assumption $\textbf{(H2)}$ and Young's inequality we obtain 
\begin{align}\label{Estim H2_a}
\E\big[\hat{X}^{n+1}_{T}(g(X^{n+1}_{T},\mu^{n}_{T})-g(X_{T}^{n},\mu^{n-1}_{T}))\big]&=\E\Big[\hat{X}^{n+1}_{T}(g(X_{T}^{n+1},\mu^{n}_{T})-g(X_{T}^{n},\mu^{n}_{T}))\Big]\nonumber\\
&+\E\big[\hat{X}^{n+1}_{T}(g(X_{T}^{n},\mu^{n}_{T})-g(X_{T}^{n},\mu^{n-1}_{T}))\big]\nonumber\\
&\geq-C_{g}^{\nu}\E[|\hat{X}^{n+1}_{T}|]\mathcal{W}_{2}(\mu^{n}_{T},\mu^{n-1}_{T})+k^{'}\E|\hat{X}^{n+1}_{T}|^{2}]\nonumber\\
&\geq (k'-\frac{C_{g}^{\nu}\epsilon}{2})\E[]|\hat{X}^{n+1}_{T}|^{2}]-\frac{2 C_{g}^{\nu}}{\epsilon }\mathcal{W}_{2}^{2}(\mu^{n}_{T},\mu^{n-1}_{T})\nonumber\\
&\geq (k'-\frac{C_{g}^{\nu}\epsilon}{2})\E[|\hat{X}^{n+1}_{T}|^{2}]-\frac{2 C_{g}^{\nu}}{\epsilon }\E[|\hat{X}^{n}_{T}|^{2}].
\end{align}
Besides using the classical linearization technics to obtain 
\begin{align*}
&\E[\displaystyle\int_{0}^{T}[\hat{Y}^{n+1}_{s}\hat{b}^{n+1}_s+\hat{X}^{n+1}_{s}\hat{h}^{n+1}_{s}(\hat{\sigma}^{n+1}_s,\hat{Z}^{n+1}_{s})]ds+\int_{0}^{T}\int_{E}\hat{K}^{n+1}_{s}\hat{\beta}^{n+1}_{s}\eta(de,ds)]\\
=\displaystyle &\E[\int_{0}^{T}[\mathcal{A}(s,U_{s}^{n+1},U_{s}^{n},\nu_{s}^{n})+\hat{Y}^{n+1}_{s}\bar{b}^{n}_{s}+\hat{X}^{n+1}_{s}\bar{h}^{n}_{s}+\hat{Z}^{n+1}_{s}\bar{\sigma}^{n}_{s}]ds\\
+&\E[\int_{0}^{T}\int_{E}\hat{K}^{n+1}_{s}\bar{\beta}^{n}_{s}\eta(de,ds)].
\end{align*}
Once again Assumption $(H2)$ and young inequality  entails
\begin{align}\label{Estim H2-b}
\hspace{0.1cm}
&\displaystyle\E\big[\int_{0}^{T}[\hat{Y}^{n+1}_{s}\hat{b}^{n+1}_{s}+\hat{X}^{n+1}_{s}\hat{h}^{n+1}_{s}+(\hat{\sigma}^{n+1}_{s},\hat{Z}^{n+1}_{s})]ds+\int_{0}^{T}\int_{E}\hat{K}^{n+1}_{s}\hat{\beta}^{n+1}_{s}\eta(de,ds)\big]\nonumber\\
\hspace{0.1cm}
&\leq -k\E[\int_{0}^{T}|\hat{Y}^{n+1}_{s}|^{2}+\|\hat{Z}^{n+1}_{s}\|^{2}+|\hat{K}^{n+1}_{s}|_{s}^{2}ds]+(\frac{C_{f}^{\nu}\alpha}{2}-k)\E[\int_{0}^{T}|\hat{Y}^{n+1}_{s}|^{2}ds]\nonumber\\
\hspace{0.2cm}
&+\E[\int_{0}^{T}[C_{h}^{\nu}| \hat{X}^{n+1}_{s}|+C_{f}^{\nu}| \hat{Y}^{n+1}_{s}|+C^{\nu}_{\sigma}\| \hat{Z}^{n+1}_{s}\|+C^{\nu}_{\beta}| \hat{K}^{n+1}_{s}|_{s}]\mathcal{W}_{2}(\nu_{s}^{n},\nu_{s}^{n-1})ds]\nonumber\\
&\leq \frac{C_{h}^{\nu}\alpha}{2}\E[\int_{0}^{T}| \hat{X}^{n+1}_{s}|^{2}ds]+\frac{C_{h}^{\nu}+C_{f}^{\nu}+C_{\sigma}^{\nu}+C_{\beta}^{\nu}}{2\alpha}\E[(|\hat{X}^{n}_{s}|^{2}+|\hat{Y}^{n}_{s}|^{2})ds]\nonumber\\
&+(\frac{C_{\sigma}^{\nu}\alpha}{2}-k)\E[\int_{0}^{T}\|\hat{Z}^{n+1}_{s}\|^{2}ds]+(\frac{C_{\beta}^{\nu}\alpha}{2}-k)\E[\int_{0}^{T}|\hat{K}^{n+1}_{s}|_{s}^{2}ds].
\end{align}
Therefore, we obtain from \eqref{Estim H2_a} and  \eqref{Estim H2-b} 
\begin{align*}
&\E\big[(k'-\frac{c_{g}^{\nu}\epsilon}{2}+\frac{\delta}{2})|\hat{X}^{n+1}_{T}|^{2}+\int_{0}^{T}(- \frac{C_{h}^{\nu}\alpha}{2}+\delta-\frac{\delta\rho}{2})|\hat{X}^{n+1}_{s}|^{2}ds+\int_{0}^{T}(\frac{C_{f}^{\nu}\alpha}{2}-k)|\hat{Y}^{n+1}_{s}|^{2}ds\big]\\
&+\E\big[\int_{0}^{T}(\frac{C_{\sigma}^{\nu}\alpha}{2}-k)|\hat{Z}^{n+1}_{s}|^{2}ds]+\E[\int_{0}^{T}(\frac{C_{\beta}^{\nu}\alpha}{2}-k)|\hat{K}^{n+1}_{s}|_{s}^{2}ds\big]\\
&\leq \E\big[ (\frac{2C_{g}^{\nu}}{\epsilon}+\frac{\delta}{2})|\hat{X}^{n}_{T}|^{2}]+\int_{0}^{T}(\frac{C_{h}^{\nu}+C_{f}^{\nu}+C_{\sigma}^{\nu}+C_{\beta}^{\nu}}{2\alpha}+\frac{\delta}{2\rho})|\hat{X}^{n}_{s}|^{2}ds\big]\\
&+\E\big[\int_{0}^{T}\frac{C_{h}^{\nu}+C_{f}^{\nu}+C_{\sigma}^{\nu}+C_{\beta}^{\nu}}{2\alpha}|\hat{Y}^{n}_{s}|^{2}ds\big].
\end{align*}

Henceforth, 
\begin{align}
\tilde{\gamma}\E[|\hat{X}^{n+1}_{T}|^{2}+\int_{0}^{T}\|\hat{U}^{n+1}_{s}\|_{s}^{2}ds]\leq \tilde{\theta}\E[|\hat{X}^{n}_{T}|^{2}+\int_{0}^{T}\|\hat{U}^{n}_{s}\|_{s}^{2}ds]
\end{align}
Choosing $\tilde{\epsilon}$ $\alpha$ and $\epsilon$ so that $\frac{\theta}{ \gamma}<1$, the inequality becomes a contraction. Thus, $(\hat{X}^{n}_{T})_{n\geq 0}$ is a Cauchy sequence in $\mathcal{L}^{2}(\Omega,\P)$ and $(\hat{X}^{n})_{n\geq 0},(\hat{Y}^{n})_{n\geq 0}, (\hat{Z}^{n})_{n\geq 0}$ and $\hat{K}^{n})_{n\geq 0})$ are Cauchy sequences respectively in $\mathcal{L}^{2}([0,T],\Omega,dt\otimes d\P)$and $\mathcal{L}^{2}_{\eta}([0,T],\Omega,dt\otimes d\nu).$ Hence, if $X, Y$, $Z$ and $K$ are the respective limits of these sequences, passing to the limit in \eqref{system 2}, we see that $(X,Y,Z,K)$ is a solution of \eqref{FBSDE}.
\end{proof}
\section{Application: Storage problem}
\label{section 3-4}
\subsection{Description of the model}
We consider a macro grid system designed to analyses energy system of $N$ nodes defined in a micro grid system called region $1,...,\Gamma$ ( power plant, electrical substations,..). We denote by $N_\gamma$ the number of nodes in group $\gamma$, so that $N = \sum_{\gamma=1}^\Gamma  N_\gamma$, and let $\pi^\gamma = N_\gamma /N$ be the ratio of the population size of region $\gamma$ to the whole population.  We shall  abusively write $i \in \gamma$ to signify that the node $i$ is in region $\gamma$.
Each node is characterized by the following:
\begin{itemize}[label=\textbullet]
\item  The capacity of the battery $S_{t}$ representing the  total {\sl energy} available in the storage device
\item The net power production of the energy (photvoltaic panels, diesel energy,..) that each nodes produces after all costs subtracted $Q_{t}$.
\end{itemize}
$-$ The power production of the energy of each nodes $i\in\left\lbrace 1,...,N\right\rbrace$ is modeled as follow \\
We denote by $Q^{i}_{t}$ the power production of a nodes $i\in\left\lbrace 1,...,N\right\rbrace $ at time $t$ is given by 
\begin{equation}
\begin{cases}
   \hspace{0.3cm} dQ_{t}^{i}=\mu^{\gamma}(t,Q^{i}_{t})dt+dM^{i}_{t}+dM^{0}_{t}\\
\hspace{0.3cm}
    Q^{r}_{0}=q^{r}_{0},
\end{cases}
\end{equation} 
where 
\begin{equation}
\begin{cases}
&dM^{i}_{t}=\sigma^{\gamma}(t,Q^{i}_{t})dB^{i}_{t}+\int_{E}\beta^{\gamma}(t,e,Q^{i}_{t})\tilde{N}^{i}(dt,de)\nonumber\\
&dM^{0}_{t}=\sigma^{\gamma^{0}}(t,Q^{0}_{t})dB^{0}_{t}+\int_{E}\beta^{\gamma^{0}}(t,e,Q^{0}_{t})\tilde{N}^{0}(dt,de)\nonumber
\end{cases}
\end{equation}
We formally consider a complete filtered probability space  $(\Omega, \mathcal{F},\P)$ on which is defined an independent stochastic processes
\begin{itemize}[label=\textbullet]
\item A standard  $B^{0},B^{1},B^{r},..,B^{N}$ Brownian motion.
\item A Poisson point process $p$ defined on $[0,T]\times\Omega\times\R|\left\lbrace 0\right\rbrace$ and $\tilde{N}_{p}$ as the associated counting measure such that 
$\hat{N}(de,dt)=n_{p}(de)dt$. We also suppose that the predictable measure $n_{p}(dx)$ is positive, finite and satisfies the following integrability condition
\begin{equation}
 \int_{\R| \left\lbrace 0\right\rbrace}(1\wedge |e|)^{2}n (de)<\infty.
\end{equation} 
\item We also consider  $N$ independent identically distributed (i.i.d) random variables $x_{0}^{i}=(s_{0}^{i},q^{i}_{0})$.
\end{itemize}
 We denote by $\F=\{\Fc_t\}$ the filtration defined by\\
 $$\hspace{0.7cm} \Fc_t = \sigma((s^i_0,  q^0_0, q^i_0), B^0_s, B^i_s, N^{0}, N^{i} \mbox{ where }  i=1, \cdots N, ~s \le t\}$$ and  by  
 $$ \hspace{0.7cm}  \Fc^0_t = \sigma( B^0_s, ~s \le t\} \mbox{ the filtration generated by } B^{0} \mbox{ and } N^{0}.$$
$-$ The Battery level $S^{i}_{t}$ in the region $\gamma$ of the node $(i)$ is controlled through a storage action  $\alpha^{\gamma,i} \in \Ac$ according  to
\begin{equation}
\label{FSDE}
\begin{cases}
S_{t}^{i}=S_{0}^{i}+\int_{0}^{t}\alpha^{i}_{s}ds\\
0\leq S_{t}^{i}\leq S_{\max}
\end{cases}
\end{equation}
In contrast with the paper of Alasseur et al \cite{clemence2017grid}, we assume that the production of energy is unpredictable. This is due to its dependence on environmental conditions such as the sun, the speed of the wind etc. which are intermittent and irregular.\\
Let's point out that, including the jumps component is essential in our analyses. In fact when the production of energy is perturbed or sudden slowdown ( winds ) , the other region can compensate this variability by optimizing the balance between production and consumption and hence stone excess electricity from a nodes to another and therefore avoiding a later expensive production . For this reason we assume that there is an exchange between the regions.\\
In fact the quantity $Q^{i}-\alpha^{i}$ is the amount of energy (excesses or bor )that the nodes can exchange with other nodes. 
We will also assume that the storage level will be enforced by a constraints. In other words, we assume that there is a maximal level for which the battery can support. $S_{max}$ is the battery's maximum instantaneous power output.\\
As in \cite{clemence2017grid}, we include a micro grid system indexed by $0$ called the "rest of the world",  which is characterized by one state variable, its    {\sl local net power production}   $Q^0_t$, and which  does not possess any storage. 
 The net production of the rest of the world is given by 
\begin{equation}
\begin{cases}
&dQ^0_t=\mu^0(t,Q^0_t) dt  + \sigma^{0}(t,Q^0_t) dB^0_t+\int_{E}\beta^{0}(t,e,Q^{0}_{t})\tilde{N}^{0}(dt,de)\\
&Q^0_0 = q^0_0.
\end{cases}
\end{equation}
\\
\underline{\emph{Electricity spot price}}\\
 We make the assumption that the electricity price per Watt-hour depends on the instantaneous demand. When the strategy   $\alpha= (\alpha^1, \cdots,\alpha^N) \in    \Ac^N$ is implemented the spot price is given by
\begin{equation}
P^{N,\alpha}_t= p\left(-  Q^0_t -  \sum_{i=1}^N \eta (Q^i_t - \alpha^i_t)\right),
\end{equation}
where $p(\cdot)$ is the exogenous inverse demand function for electricity 
and $\eta$ is a scaling parameter which weights the contribution of each individual node $i$ to the whole system. \\ 
We assume that $p$ is non deceasing function 
The inverse demand function is usually taken 
Since the energy model is modeled through a macro grid system of a large number of nodes. we are naturally leads to express our model in term of mean field problem . Hence, we will consider the following spot price \begin{equation}
P^{N,\alpha}_t= p\left( - Q^0_t -  \sum_{i=1}^N  \frac{1}{N}(Q^i_t - \alpha^i_t)\right).
\end{equation}
where  $\frac{1}{N} \sum_{i=1}^N (Q^i_t - \alpha^i_t)$
 is the  averaged net injections and  $\eta = 1/N$.\\
 \\
 \underline{\emph{The control problem}}\\
We consider  a finite time horizon $T>0$.
An admissible control  $\alpha = (\alpha^1, \cdots, \alpha^N)$is defined to be a square integrable $(\mathcal{F}_{t})_{0\leq t\leq T}$-adapted process with values in $\R^{m}$. Any nodes $i$ controls it level of reserve through $\alpha^{i}$ in order to minimize the following cost function. 
\begin{equation}
J^{i,\gamma,N}(\alpha)=\E[
	\int_0^T P^{N,\alpha}_t.\left( \alpha^i_t - Q^i_t\right) + L^\gamma_T(Q^{i}_t, \alpha^i_t)+ L_S(S^{i, \alpha^i}_t, \alpha^i_t)dt + g(S^{i, \alpha^i}_T)]
\end{equation}
The cost function $J$ depend on the revenue of the storage ( bought, or rented) and  also  on the total cost of the system. \\
$\bullet$ $L_{S}$ represent the running storage cost as investment in capacities , operation and maintenance. \\
$\bullet$ $L_{T}$ is the volumetric charge. This electricity cost is closely related to the power that system requires in peak hours and hence produce enough power to satisfy highest level of peak demand. \\
\underline{\emph{The rest of the world}  }
The rest of the world incurs  only  energy and transmission costs 
\begin{equation}
	J^{0,N}(\alpha)=\E\left[
	\int_0^T - P^{N,\alpha}_t . Q^0_t+L^0_T(Q^{0}_t, 0)dt\right]
		\end{equation}

\subsection{Reformulation: Mean field type control problem}
In this section we consider  on the filtered probability space $(\Omega, \Fc, \P, \F)$,  $\Gamma$ standard brownian motions $B^\gamma, \gamma =1, \cdots, \Gamma$ which are mutually independent and independent from the Brownian filtration $\F^0$.
 We shall use the following notation. If $\xi = \{\xi_t \}$ is an $\F$-adapted process, then $\bar \xi = \{ \bar \xi_t\}$ denotes the process defined by :~$\bar \xi_t := \E[\xi_t | \Fc^0_t]$.
\bigskip

Let $x_0=(s_0, q_0) = \left( x_0^{\gamma}=(s_0^{\gamma}, q_0^{\gamma})\right)_{ 1\le {\gamma} \le \Gamma}$ be a random vector which is independent from $\F^0$. Let $Q^0$ and $Q^\gamma$ be the processes defined by
\begin{align}
\label{jump-q}
\vspace{0.1cm}
&\displaystyle Q^{\gamma}= q_0^{\gamma}+\int_0^t\mu^{\gamma}(u,Q_{u}^{\gamma}) du  +\int_0^t   \sigma^{\gamma}(u,Q_{u}^{\gamma}) dB^{\gamma}_u +\int_0^t   \sigma^{{\gamma},0}(u,Q_{u}^{\gamma}) dB^0_u\nonumber\\
\vspace{0.2cm}
&\qquad\,\,\quad+\int_{0}^{t}\int_{E}\beta^{\gamma}(u,e,Q^{\gamma}_{u})\tilde{N}(du,de)+\int_{0}^{t}\int_{E}\beta^{\gamma^{0}}(u,e,Q^{0}_{u})\tilde{N}^{0}(du,de)\nonumber
\\
&Q^0_t=q_0^{0}+  \int_0^t \mu^r(u,Q^0_t) du +  \int_0^t \sigma^{0}(u,Q^0_u) dB^0_u+\int_{0}^{t}\int_{E}\beta^{\gamma}(u,e,Q^{0}_{u})\tilde{N}^{0}(du,de).
	\end{align}

If  $\bar \nu =  (\bar \nu^{1}, \cdots, \bar \nu^{\Gamma})$
 is an $\F^0$-adapted   $\R^\Gamma$-valued process, we denote
 	\begin{equation}
	\label{price:MFC}
	P^{\bar \nu}_t=	p\left(-  Q^0_t   - \sum_{{\gamma} \in \Gamma}\pi^{\gamma}\left(\E[ Q^{{\gamma}}_t | \Fc^0_t] - \bar \nu^{{\gamma}}_t \right)\right).
        \end{equation}

We now consider the following cost functions,  for any control process $\alpha=(\alpha^1, \cdots, \alpha^\Gamma)$ and for each $\gamma =1, \cdots, \Gamma$, 
	\begin{equation}
	 J^{\gamma}_{x_0}(\alpha^\gamma, \bar{ \nu})=
	\E\int_0^T \left[ P^{\bar \nu}_t (\alpha^\gamma_t - Q^\gamma_t) + 
	L^\gamma_T(Q^\gamma_t, \alpha^\gamma_t) + L_S(S^\gamma_t, \alpha^\gamma_t)\right]dt + \E\left[ g(S^\gamma_t)\right].
	\end{equation}
	\begin{equation}
	J^{C}_{x_0}(\alpha)=\E\int_0^T [-P^{\bar{\alpha}}_t Q^0_t +L^0_T(Q^0_t, 0 )]dt	+ \sum_{\gamma=1}^\Gamma \pi^\gamma J_{x_0}^\gamma(\alpha^\gamma, \bar \alpha_t).
	\end{equation}
	where
	$S^{{\gamma}}_t = s_0^{\gamma} + \int_0^t \alpha^\gamma_u du.$


Following \cite{clemence2017grid}, we characterize the solution and the control process $\alpha$ for a fixed mean field control $\bar\nu$ and prove that if the conditional expectation of the control process  $\alpha$ is equal to $ \bar \nu $ then necessary we have the Nash equilibrium . \\

\begin{Definition}[Mean field Nash equilibrium]
Let $x_0=(s_0,q_0)$ be  a random vector independent from $\F^0$. We say that $\alpha^\star = \{\alpha^{\gamma,\star},1\le {\gamma} \le \Gamma\}$ is a mean field Nash equilibrium   if,   for each $\gamma$, $ \alpha^{\gamma,\star}$ minimizes the function
$\alpha^\gamma \mapsto J^{\gamma}_{x_0}(\alpha^\gamma, \{ \E[\alpha^\star_t | \Fc^0_t]\})$.
\end{Definition}

\begin{Definition}[Mean field optimal control]
Let $x_0=(s_0,q_0)$ be  a random vector independent from $\F^0$. We say that $\hat \alpha= \{\hat \alpha^{\gamma},1\le {\gamma} \le \Gamma\}$ is a mean field optimal control   if,    $ \hat \alpha$ minimizes the function
$\alpha \mapsto J^{C}_{x_0}(\alpha)$.
\end{Definition}
\underline{\textbf{Characterization of  mean field Nash equilibrium}}
\begin{Proposition}
Let $\bar \nu$ be a given $\F^0$-adapted   $\R^\Gamma$-valued process. \\
Then there exists a unique  control $ (\alpha^{1,\star}, \cdots, \alpha^{\Gamma, \star}) = \alpha^\star(\bar \nu, x_0)$ such that\\
$\bullet$ For each $\gamma\in 1,.., \Gamma $, $\alpha^{\gamma,\star}$ minimizes the function   $\alpha^\gamma \mapsto J^{\gamma}_{x_0}(\alpha^\gamma, \bar \nu)$.\\
$\bullet$ If $(S^{\gamma,\star}, Q^\gamma)$  is the state process  corresponding to the initial data condition $x^\gamma_0$, to  the control  $\alpha^{\gamma,\star}$,  and to the dynamic above, then there exists a unique adapted solution \\$(Y^{\gamma, \star}, Z^{0,\gamma, \star}, Z^{\gamma, \star}, V^{\gamma, \star})$ of the BDSE with jumps
	\begin{align}
	\label{nash-FBSDE}
Y^{\gamma, \star}_t&=\partial_s g(S^{\gamma,\star}_T) +\int_{0}^{T}\partial _s L_S(S^{\gamma,\star}_t, \alpha^{\gamma,\star}_t)dt + \int_{0}^{T}  Z^{0,\gamma, \star}_t dB^0_t + Z^{\gamma,\star}_t dB^\gamma_t
\nonumber\\&+\int_{0}^{T}\int_{E}V^{\gamma, \star}(e)\tilde{N}(dt,de)+\int_{0}^{T}\int_{E}V^{0,\gamma, \star}(e)\tilde{N}^{0}(dt,de)
\end{align}
satisfying the coupling condition
	\begin{equation}\label{nash coupling condition}
	Y^{\gamma,\star}_t + 
	P^{\bar \nu}_t +
	\partial_\alpha L^\gamma_T(Q^\gamma_t,\alpha^{\gamma,\star}_t) + \partial_\alpha L_S(S^{\gamma,\star}_t,\alpha^{\gamma,\star}_t)=0
	\end{equation}
Conversely, assume that there exists $(\alpha^{\gamma,\star}, S^{\gamma,\star}, Y^{\gamma, \star}, Z^{0,\gamma,\star},Z^{\gamma,\star}, V^{\gamma, \star})$ which satisfy the coupling condition  \eqref{nash coupling condition} as well as the FBSDEJ, then $ \alpha^{\gamma,\star}$ is the optimal control minimizing $J^{\gamma}_{x_0}(\alpha^\gamma, \bar \nu)$ and $S^{\gamma,\star}$ is the optimal trajectory.\\
$\bullet$ If in addition:$\forall \gamma =1, \cdots, \Gamma$ $\E\left[  \alpha^{\gamma,\star}_t| \Fc^0_t\right] = \bar \nu^{\gamma,0}_t,$ then $ \alpha^\star$ is a mean field Nash equilibrium.

\end{Proposition}
\begin{proof}
$\bullet$ Since the dynamic programming principal does not work in this context,  we derive the optimality condition directly from the Euler Lagrange equation 
\begin{equation}
d_\beta J^{\gamma}_{x_0}( .,\bar \nu):=0
\end{equation}
We start by computing the functional directional derivative of $J^{\gamma}_{x_0}( .,\bar \nu)$
\begin{align}
d_\beta J^{\gamma}_{x_0}( .,\bar \nu)&=\E\left[\int_0^T [P^{\bar \nu}_u +\partial_\alpha L^\gamma_T(Q^\gamma_u,\alpha^\gamma_u) + \partial_\alpha L_S(S^\gamma_u,\alpha^\gamma_u)+\partial_s L_{S}(S^{\gamma}_u, \alpha^\gamma_u) ]\beta_u  du \right]\\
&+ \E\left[\tilde S^\beta_T \partial_s g(S^{\gamma}_T)\right]
	\end{align}
Hence, 	there exists a unique optimal control $\alpha^{\gamma,\star} =\alpha^{\gamma,\star}(\bar \nu, x_0)$ such that 
	\begin{equation}
\E\left[\int_0^T [P^{\bar \nu}_u +\partial_\alpha L^\gamma_T(Q^\gamma_u,\alpha^\gamma_u) + \partial_\alpha L_S(S^\gamma_u,\alpha^\gamma_u)+\partial_s L_{S}(S^{\gamma}_u, \alpha^\gamma_u) ]\beta_u  du \right]
+ \E\left[\tilde S^\beta_T \partial_s g(S^{\gamma}_T)\right]=0
	\end{equation}
We denote by $S^{\gamma,\star}$ the optimal trajectory associated to  $\alpha^{\gamma,\star}$. 
Applying It\^o tanaka formula to $S^\beta_t Y^{\gamma,\star}_t$ we get 
\begin{align*}
\tilde S^\beta_t Y^{\gamma,\star}_t& = \tilde S^\beta_T Y^{\gamma,\star}_T+\int_{t}^{T} Y^{\gamma,\star}_s \beta_s ds -\tilde S^\beta_s f()ds\\
&+\int_{0}^{T}  \tilde S^\beta_sZ^{0,\gamma, \star}_s dB^0_s + \int_{t}^{T} \tilde S^\beta_s Z^{\gamma,\star}_s dB^\gamma_s+\int_{0}^{T}\int_{E}\tilde S^\beta_s V_{s}^{\gamma, \star}(e)\tilde{N_{p}}(dt,de).
 \end{align*}
 Taking the conditional expectation of the above the equation , we obtain 
 \begin{equation}
 \E\left[ \tilde S^\beta_T Y^{\gamma,\star}_T\right]=\E[\int_{t}^{T} Y^{\gamma,\star}_s \beta_s ds -\tilde S^\beta_s f(S^{\gamma,\star}_t, \alpha^{\gamma,\star}_t)ds].
 \end{equation}
 Using the above Euler optimality condition yield's to 

\begin{equation}
	\E\left[\int_0^T 
	\left( Y^{\gamma,\star}_s +P^{\bar \nu}_s +\partial_\alpha L^\gamma_T(Q^\gamma_s,\alpha^{\gamma,\star}_s) + \partial_\alpha L_S(S^{\gamma,\star}_s,\alpha^{\gamma,\star}_s)
	\right) \beta_sds
	\right]=0.
	\end{equation}
	$\bullet$ If $(\alpha^{\gamma,\star},S^{\gamma,\star}, Y^{\gamma, \star}, Z^{0,\gamma,\star},Z^{\gamma,\star},V^{\gamma,\star})$ is a solution of the following coupled Forward-Backward SDE with jumps. 
	\begin{equation}
	    \begin{cases}
	    \vspace{0.1cm}
	   & Y^{\gamma, \star}_t=\partial_s g(S^{\gamma,\star}_T) +\int_{0}^{T}\partial _s L_S(S^{\gamma,\star}_t, \alpha^{\gamma,\star}_t)dt + \int_{0}^{T}  Z^{0,\gamma, \star}_t dB^0_t + Z^{\gamma,\star}_t dB^\gamma_t\\
	   \vspace{0.2cm}
	   &\qquad+\int_{0}^{T}\int_{E}V^{\gamma, \star}(e)\tilde{N}(dt,de)+\int_{0}^{T}\int_{E}V^{0,\gamma, \star}(e)\tilde{N}^{0}(dt,de).\\
	   \vspace{0.1cm}
	&Q^{\gamma}= q_0^{\gamma}  + \int_0^t  
	\mu^{\gamma}(u,Q^{\gamma}) du  +\int_0^t   \sigma^{\gamma}(u,Q^{\gamma}) dB^{\gamma}_u +\int_0^t   \sigma^{{\gamma},0}(u,Q^{\gamma}) dB^0_u\\
	\vspace{0.2cm}
	&\qquad+\int_{0}^{t}\int_{E}\beta^{\gamma}(u,e,Q^{\gamma}_{u})\tilde{N}(du,de)+\int_{0}^{t}\int_{E}\beta^{\gamma,0}(u,e,Q^{\gamma, 0}_{u})\tilde{N}^{0}(du,de)\\
	&S^{{\gamma}}_t = s_0^{\gamma} + \int_0^t \alpha^\gamma_u du
	    \end{cases}
	\end{equation}
	$$ \hbox{ with }Y^{\gamma,\star}_t + 
	P^{\bar \nu}_t +
	\partial_\alpha L^\gamma_T(Q^\gamma_t,\alpha^{\gamma,\star}_t) + \partial_\alpha L_S(S^{\gamma,\star}_t,\alpha^{\gamma,\star}_t)=0.$$
	Then $ \alpha^{\gamma,\star}\in \argmin J^{\gamma}_{x_0}(\alpha^\gamma, \bar \nu)$ with $S^{\gamma,\star}$ is the optimal trajectory. and Finally if $\forall \gamma =1, \cdots, \Gamma,$
 $E\left[  \alpha^{\gamma,\star}_t| \Fc^0_t\right] = \bar \nu^{\gamma,0}_t$
then $ \alpha^\star$ is a mean field nash equilibrium.\\
	\end{proof}
\underline{\textbf{Characterization of  mean field optimal controls}}
\begin{Proposition}\label{prop: charac mfc} 
Assume that $\hat \alpha=(\hat \alpha^1, \cdots, \hat \alpha^\Gamma)$  minimizes the functional $J^{C}_{x_0}(\alpha)$, and denote by  $\hat S = (\hat S^\1, \cdots, \hat S^\Gamma)$ is the corresponding controlled trajectory.  Then there exists a unique adapted solution $(\hat Y = (\hat Y^1, \cdots \hat Y^\Gamma_t),\hat Z = (\hat Z^1, \cdots, \hat Z^\Gamma),\hat Z^0 = (\hat Z^{0,1}, \cdots, \hat Z^{0,\Gamma}))$ of the BDSE
	\be\label{eq: bsde mfc}
	\left\{
	\begin{array}{l l l}
d\hat Y^\gamma_t&=& - \partial_s L_S(\hat S^\gamma_t, \hat \alpha^\gamma_t) dt  +  \hat Z^{0,\gamma}_t dB^0_t + \hat Z^\gamma_t dB^\gamma_t\int_{E}V^{\gamma }(e)\tilde{N}(ds,de)+\int_{E}V^{0,\gamma }(e)\tilde{N}^{0}(ds,de)\\
	\hat Y^\gamma_T&=&\partial_sg(\hat S^\gamma_T)
	\end{array}
	\right.
	\ee 
satisfying the coupling condition:~for all $\gamma= 1, \cdots, \Gamma$
	\be
	\nonumber
	0&=& \hat Y^\gamma_t 
	+\partial_\alpha L^\gamma_T(Q^\gamma_t, \hat \alpha^\gamma_t) 
	+ \partial_\alpha L_S(\hat S_t, \hat \alpha^\gamma_t) + P^{\bar{\hat  \alpha}}_t  \\
	\label{eq: coupling mfc}
	&& \hspace{7mm}
	 -  p'\left(- Q^{0}_t - 
	 \Pi_{\Gamma} \cdot \left(\bar Q_t - \bar {\hat \alpha}_t \right)
	  \right)
	\left( - Q^{0}_t - \Pi_{\Gamma} \cdot \left(\bar Q_t - \bar{\hat  \alpha}_t \right)  \right) 
	\ee

	with $\bar{\hat  \alpha}_t = \E[\hat \alpha_t | \Fc_t^0]$ and $\Pi_{\Gamma}=(\pi_1, \cdots, \pi_\Gamma)^T$.

\no Conversely, suppose $(\hat S, \hat \alpha, \hat Y, \hat Z^0, \hat Z)$ is an adapted solution to the forward backward system \reff{FSDE}-\reff{eq: bsde mfc}, with the coupling condition \reff{eq: coupling mfc}, then $\hat \alpha$ is the optimal control minimizing $J^{\rm MFC}_{x_0}(\alpha)$ and $\hat S$ is the optimal trajectory. 
\end{Proposition}
\subsection{Explicit solution of the MFC with 1 region} \label{explicit-exple}
In this section, we provide an example where an explicit solution of the MFC problem is obtained.
We consider a linear pricing rule of the following form     
\be\label{eq: linear pricing rule}
p(x)=p_0 + p_1 x.
\ee
 The storage cost $L_{S}$ is defined by: For $A_1<0, A_2>0, C<0$,$$L_S(s,\alpha) =A_1 s + \frac{A_2}{2} s^2 + \frac{C}{2} \alpha^2.$$
 For some given positive constant $\{K^{\gamma}\}_{\gamma=1}^{\Gamma}$, the transmission cost is defined by   $$ L^\gamma_T(q,\alpha) =\frac{K^\gamma}{2}\left( q - \alpha\right)^2.$$
 For some constants $B_{1}$ and $B_{2}>0$, the terminal cost $$ g(s) =\frac{B_2}{2}\left(s-\frac{B_1}{B_2}\right)^2.$$

Now, we will consider the simple case of one region, i.e. when $\pi =1$. we aim to find an explicit solution to the MFC problem associated to the linear quadratic case. 

\no {\bf Step 1.}~ In this first step, we use the forward backward system (\eqref{eq: bsde mfc}, \eqref{FSDE}) and the coupling condition \eqref{eq: coupling mfc} in order to get the optimal control $\bar{\alpha}$ and the optimal trajectory $\bar{S}$ associated to one node in this region. We have
\begin{equation}
\label{FBSDE-explicit}
\begin{cases}
&d\bar S_t= \bar \alpha_t dt,~~\bar S_0 = 0,\\
&d \bar Y_t= - (A_2\bar S_t  + A_1)dt + \bar Z^{0}_t dB^0_t+\int_{E}\bar{V}_{s}^{\gamma, \star}(e)\tilde{N}(ds,de),~~\bar Y_T = B_2 \bar S_T - B_1.
\end{cases}
\end{equation}
To find the optimal control $\bar{\alpha}$, we use firstly the coupling condition \eqref{eq: coupling mfc} to obtain
\begin{equation}
\bar{Y}_{t}-K(\bar{Q}_{t}-\bar{\alpha}_{t})+C\bar{\alpha}_{t}+P^{\bar{\alpha}}_{t}-p'(-Q^{0}_{t}-\bar{Q}_{t}+\bar{\alpha}_{t})(-Q^{0}_{t}-\bar{Q}_{t}+\bar{\alpha}_{t})=0,
\end{equation}
where $Q^{0}$ and $Q$ are defined by \eqref{jump-q}.
Now, using Proposition $(3.3)$ in \cite{clemence2017grid} and the linear form of $p$ in \eqref{eq: linear pricing rule}, we obtain
\begin{equation}
P^{\bar{\alpha}}_{t}=p_{0}+2p_{1}(-Q^{0}_{t}-\bar{Q}_{t}+\bar{\alpha}_{t}),
\end{equation}
and we obtain the following expression of the optimal control $\bar{\alpha}$:
\begin{align*}
\bar{\alpha_{t}}&=-\frac{1}{K+C+p_{1}}\Big(\bar{Y}_{t}+p_{0}-p_{1}Q^{0}_{t}-(p_{1}+K)\bar{Q}_{t}\Big)\\
&=-\Delta(\bar{Y}_{t}+b_{t}),
\end{align*}
where $\Delta=\frac{1}{K+C+p_{1}}$ and $b_{t}=p_{0}-p_{1}Q^{0}_{t}-(p_{1}+K)\bar{Q}_{t}$.\\

\noindent We expect the solution of the FBSDE \eqref{FBSDE-explicit} to be affine. It has the following 
form: 
\begin{equation}
\bar{Y}_{t}=\bar{\phi}_{t}\bar{S}_{t}+\bar{\psi}_{t},
\end{equation} 
where $\phi$ and $\psi$ are deterministic functions.
Computing $d\bar{Y}_{t}$ from this expression, we obtain 
\begin{equation}
d\bar{Y}_{t}=\bar{S}_{t}(-\Delta \bar{\phi}_{t}^{2}+\dot{\bar{\phi}}_{t})dt-\Delta\bar{\phi}_{t}(\bar{\psi}_{t}dt+ b_{t})dt+\dot{\bar{\psi}}_{t}.
\end{equation}
Identifying the two expressions of $d\bar{Y}_{t}$ we get, in one hand, that 
\begin{equation}
\label{Riccati}
\dot{\bar{\phi}}_{t}-\Delta {\bar{\phi}}^{2}_{t}+A_{2}=0,\quad \bar{\phi}(T)=B_{2}
\end{equation}
which is a Riccati equation.

In the other hand, we obtain that $\psi$ is the unique solution of the BSDE 
\begin{equation}
d\bar{\psi}_{t}=\Delta\bar{\phi}_{t}(\bar{\psi}_{t}+ b_{t})dt-A_{1}dt+\bar{Z}^{0}_{t}dB_{t}^{0}+\int_{E}\bar{V}_{s}^{\gamma, \star}(e)\tilde{N}(ds,de).
\end{equation}
Consequently, substituting $b_{t}$,  we get the following BSDE 
\begin{equation}
d\bar{\psi}_{t}=\Delta\bar{\phi}_{t}(\bar{\psi}_{t}+ \bar{P}_{t})dt+\bar{Z}^{0}_{t}dB_{t}^{0}+\int_{E}\bar{V}_{s}^{\gamma, \star}(e)\tilde{N}(ds,de), \quad \psi_{T}=-B_{1}.
\end{equation}
 This allows us to find the expression of the approximated electricity price. In fact, in one hand we have
\begin{align*}
\bar{\psi}_{t}&=\Delta\bar{\phi}_{t}(\bar{\psi}_{t}+ \bar{P}_{t})dt+\bar{Z}^{0}_{t}dB_{t}^{0}+\int_{E}\bar{V}_{s}^{\gamma, \star}(e)\tilde{N}(ds,de)\\
&=\Delta\bar{\phi}_{t}(\bar{Y}_{t}-\bar{\phi}_{t}\bar{S}_{t}+\bar{P}_{t})dt+\bar{Z}^{0}_{t}dB^{0}_{t}+\int_{E}\bar{V}_{s}^{\gamma, \star}(e)\tilde{N}(ds,de)\\
&=\Delta\bar{\phi}_{t}\bar{Y}_{t}dt-\Delta(\bar{\phi}_{t})^{2}\bar{S}_{t} dt+\int_{E}\bar{V}_{s}^{\gamma, \star}(e)\tilde{N}(ds,de).
\end{align*}
In the other hand, we have that $d\bar{\psi}_{t}=d\bar{Y}_{t}-\dot{\bar{\phi}}_{t}S_{t}dt-\bar{\phi}_{t}dS_{t}$. So, we obtain 
\begin{align*}
-(A_{2}\bar{S}_{t}dt+A_{1})dt-\Delta\bar{\phi}_{t}\bar{P}_{t}dt=\dot{\bar{\phi}}_{t}\bar{S}_{t}dt+\bar{\phi}_{t}d\bar{S}_{t}+\Delta\bar{\phi}_{t}\bar{Y}_{t}-\Delta(\bar{\phi}_{t})^{2}\bar{S}_{t} dt.
\end{align*}
Finally, Using the Riccati equation \eqref{Riccati} in the equation above, we obtain directly the following price expression 
\begin{equation*}
\bar{P}_{t}=-\frac{A_{1}}{\Delta \bar{\phi}_{t}}+b_{t}.
\end{equation*}
As it can be seen, $\bar{\Psi}$ is the solution of linear BSDE with jumps. So it has the following expression 
\begin{equation*}
\bar{\Psi}_{t}=\mathbb{E}[-\Gamma_{t,T}B_{1}+\int_{t}^{T}\Gamma_{t,u}\Delta\bar{\phi}_{u}\bar{P}_{u}du|\mathcal{F}^{0}_{t}],
\end{equation*}
where $\Gamma_{t,T}$ is the adjoint process and in this case, it is the solution of 
\begin{equation*}
d\Gamma_{t,s}=\Gamma_{t,s}\Delta \bar{\phi}_{s}ds, 
\end{equation*}
which is $\Gamma_{t,s}=\exp(\int_{t}^{s}\Delta \bar{\phi}_{s}ds)$.
Consequently, $\Psi_{t}$ is given by 

\begin{equation*}
	\bar \Psi_t= - B_1 \exp\left\{ - \int_t^T \Delta \bar \phi(u) du\right\}-\E\left[
	 \int_t^T  \Delta \bar \phi(u)  \exp\left\{ - \int_t^u \Delta \bar \phi(s) ds\right\} \bar P_u du | \Fc^0_t 
	 \right].
	\end{equation*}
The function $\bar \phi$  is given by
	\b*
	\bar \phi(t)&=&- \frac{\rho}{\Delta}\; \frac{
	e^{- \rho (T-t)}(- B_2 \Delta + \rho) - e^{ \rho (T-t)}(B_2 \Delta + \rho)
	}
	{
	e^{- \rho (T-t)}(- B_2 \Delta + \rho) + e^{ \rho (T-t)}( B_2 \Delta + \rho)
	}\;~~\mbox{with}~~\rho := \sqrt{A_2 \Delta},
	\e*	
Now, to find $\bar{S}_{t}$, it suffices to solve the following  simple EDO 
\begin{equation}
d\bar{S}_{t}=\big[-\Delta\bar{\phi}_{s}\bar{S}_{s}-\Delta (\Psi_{s}+\bar{P}_{s}+\frac{A_{1}}{\Delta \bar{\phi}_{s}})\big]ds,
\end{equation}
for which the solution is given 
\begin{equation}
\bar{S}_{t}=\bar{S}_{0}\exp(\int_{0}^{t}-\Delta\bar{\phi}_{s}ds)-\Delta\int_{0}^{t}\exp(\int_{u}^{t}-\Delta\bar{\phi}_{s}ds)(\Psi_{u}+\bar{P}_{u}+\frac{A_{1}}{\Delta \bar{\phi}_{u}})du.
\end{equation}
As $\bar{S}_{0}=0$, the solution is then 
\begin{equation*}
\bar S_t=
-\Delta \int_0^t    \exp\left\{ - \int_u^t \Delta \bar \phi(s) ds\right\} \left(
\bar P^{}_u + \bar \Psi_u + \frac{ A_1}{\ \Delta \bar \phi(u)}\right)
 du .
\end{equation*}

\no {\bf Step 2.}~Once we obtain all the optimal elements of one node in the first step, we use the FBSDE \eqref{nash-FBSDE} and the coupling condition \eqref{nash coupling condition} to find the optimal objects associated to  one region containing a number of identical nodes.

\noindent Using the FBSDE \eqref{nash-FBSDE}, we have 
\b*
	dS_t&=& - \delta \left( Y_t + P_t + \frac{A_1}{\delta \phi(t)}\right) dt, ~~S_0 = s_0,\\
	dY_t&=& - (A_2 S_t + A_1) dt + Z^0_t dB^0_t + Z_t dB_t+\int_{E}V_{s}(e) \tilde{N}(ds,de),~~Y_T = B_2 S_T - B_1. \\
	\e*
Again, we look at a solution of the form
	\b*
	Y_t &=& \varphi(t) S_t + \psi_t,
	\e*	
and using the coupling condition \eqref{nash coupling condition}, we obtain the expression the optimal control $\alpha$. In fact,  
\begin{align*}
Y_{t}-K(Q_{t}-\alpha_{t})+C\alpha_{t}+P^{\bar{\alpha}}_{t}-p_{1}(-Q^{0}_{t}-Q_{t} -\alpha_t)=0
\end{align*}
where $P^{\bar{\alpha}}_t=p_0+2p_1(-Q^{0}_{t}-\bar{Q}_{t}+\bar{\alpha}_t)$
Then 
\begin{align*}
\alpha_{t}&=-\delta\big(Y_{t}+p_0-KQ_t-2p_1(Q^{0}_t+\bar{Q}_{t}-\bar{\alpha}_t)+\frac{A_1}{\delta \phi_t}\Big)\\
&=-\delta(Y_t+P_t+\frac{A_1}{\delta \phi_t}),
\end{align*}
where  $\delta=\frac{-1}{K+C}$, $P_t=p_0-KQ_t-2p_1(Q^{0}_t+\bar{Q}_{t}-\bar{\alpha}_t)-\frac{A_1}{\delta \phi_t}$


Once again, we aim to find a solution to the FBSDE  above which has the following form $Y_{t}=\varphi_{t}S_{t}+\psi_{t}$
	\b*
	Y_t &=& \varphi(t) S_t + \psi_t,
	\e*	
		
where $\varphi$ and $\psi$ can be explicitly calculated in the same ways as before. In fact, $\psi$ is the solution of the following 
\begin{equation}
d\psi_t=-\delta \varphi_t(\psi_t dt+P_t)+Z^{0}_tdB^{0}_t+Z_{t}dB_t+\int_{E}V(e)\tilde{N}(de,ds)
\end{equation}
and $\varphi_t$ satisfies the following Riccati equation 
\begin{equation}
\delta \varphi^{2}_{t}-\dot{\varphi}_{t}-A_2=0
\end{equation}
whose solution is given by
\b*
	\varphi(t)&=&- \frac{\rho}{\delta}\; \frac{
	e^{- \rho (T-t)}(- B_2 \delta + \rho) - e^{ \rho (T-t)}(B_2 \delta + \rho)
	}
	{
	e^{- \rho (T-t)}(- B_2 \delta + \rho) + e^{ \rho (T-t)}( B_2 \delta + \rho)
	}\;~~\mbox{with}~~\rho := \sqrt{A_2 \delta},
	\e*
	
As it can be seen, $\psi_{t}$ is the solution of the following BSDE driven by a $2$-dimensional Brownian motion

\begin{equation}
d\psi_t=-\delta \varphi_t(\psi_t+P_t)dt+\tilde{Z}_{t}d\tilde{B}_{t}+\int_{E}V(e)\tilde{N}(de,ds)
\end{equation}
where $\tilde{Z}_{t}=(Z^{0}_{t},Z_{t})$ and $\tilde{B}_{t}=(B^{0}_{t},B_{t})$.

Finally 
\begin{equation*}
S_t= s_0 \exp\left\{ - \int_0^t \delta  \varphi(u) du\right\} 
- \delta \int_0^t  \exp\left\{ - \int_u^t \delta \varphi(s) ds\right\} \left( 
P_u +  \psi_u +\frac{A_1}{\delta \phi(u)}
\right) du .
\end{equation*}	

\section{Appendix}
In this section, we extend some of the results of Hamadene \cite{hamadene1998backward} concerning FBSDEs in the Brownian setting  to the case of jumps.
 Let us note that arguments of proof are close to the one used by Hamadene in \cite{hamadene1998backward} with some minor modifications due to jumps setting. However, we still provide the proof of existence since it will be needed in the construction of mean field FBSDEJ's solution.\\
We look for the solution of the following fully coupled Forward-Backward SDEs with jumps 

\begin{equation}\label{Appendix system}
(\textbf{S})
\begin{cases}
&X_{t}=X_{0}+\displaystyle\int_{0}^{t}b_{s}(X_{s},Y_{s},Z_{s},K_{s})ds+\int_{0}^{t}\sigma_{s}(X_{s},Y_{s},Z_{s},K_{s}(e))dW_{s}\\
\vspace{0.3cm}
& \qquad +\int_{0}^{t}\int_{E}\beta_{s}(X_{s^{-}},Y_{s^{-}},Z_{s},K_{s}(e))\tilde{\pi}(ds,de),\hspace{0.2cm}0 \leq t\leq T,\, \P\hbox{-a.s.}\\
&Y_{t}=g(X_{T},\mathbb{P}_{X_{T}})-\int_{t}^{T}h_{s}(X_{s},Y_{s},Z_{s},K_{s})ds-\int_{t}^{T}Z_{s}dW_{s}-\int_{t}^{T}\int_{E}K_{s}(e)\tilde{\pi}(ds,de),
\end{cases}
\end{equation}
\noindent For $u=(x,y,z,k)$ and $u'=(x',y',z',k') \in \mathbb{R}^{m+m+m\times m}$, we define the function $\mathcal{A}$ as follows
\begin{align*}
\mathcal{\bar A}(t,u,u')&=\big[f(s,x,y,z,k)-f(s,x',y',z',k')\big](y-y')+\big[h(s,x,y,z,k)-h(s,x',y',z',k')\big](x-x')\\
&+[\sigma(s,x,y,z,k)-\sigma(s,x',y',z',k^{'})](z-z')\\
&+\int_{E}(\beta(s,x,y,z,k)-\beta(s,x',y',z',k'))(k-k')(e)\eta(ds,de).
\end{align*} 
We assume the following Assumption 
\begin{Assumptions}\label{bar H1}
\begin{equation*}
(\bar{H}1)
\begin{cases} 
(i)\mbox{ There exists } k>0, \mbox{ s.t } \forall t \in[0,T], \nu\in \mathbb{M}_{1}(\mathbb{R}^{m}\times \mathbb{R}^{m}), u, u' \in \mathbb{R}^{m+m+m \times m}, \,\,\,\,\\
 \qquad \qquad \mathcal{\bar A}(t,u,u')\leq -k|x-x^{'}|^{2}, \mathbb{P}\hbox{-a.s.}\\
(ii)\mbox{ There exists } k'>0, \mbox{ s.t } \forall \nu\in \mathbb{M}_{1}(\mathbb{R}^{m}\times \mathbb{R}^{m}),x,x'\in \mathbb{R}^{m}\\
\qquad \qquad(g(x,\nu)-g(x',\nu)).(x-x')\leq  k'|x-x'|^{2}, \mathbb{P}\hbox{-a.s.}\\
(iii)\,\, min\{k,k'\}\geq C(\tilde{C}+1).
\end{cases}
\end{equation*}
\end{Assumptions}
We also assume that 
\begin{equation*}
    \begin{cases}
    &(i)\quad  b,\sigma, g \mbox{ and } f \mbox{ satisfy the uniformly Lipschitz condition with respect to } (x,y,z,k).\\
    &(ii)\quad g \mbox{ is uniformly Lipschitz with respect to } $x$.
    \end{cases}
\end{equation*}

\begin{Proposition}\label{ Appendix existence H1}
Under Assumption $(\bar H1)$, there exists a unique solution $U=(X,Y,Z,K)$ of the FBSDE with jumps \eqref{Appendix system}
\end{Proposition}
 \begin{proof}
 In the following, we will use in proofs the notation $\tilde C$ to denote a generic constant that may change from line to line and that depends in an implicit way on $T$ and the Lipschitz constants.\\
The key point of the proof is to  consider a sequence $U^{n}=(X^{n},Y^{n},Z^{n},K^{n})$ of processes defined recursively by :
$(X^{0},Y^{0},Z^{0},K^{0})=(0,0,0,0)$ and for $n\geq 1$, $U^{n}=(X^{n},Y^{n},Z^{n},K^{n})$ satisfies: $\forall\, 0\leq t\leq T$ and  $\delta \in ]0,1]$
\begin{align}\label{sequence FBSDE}
\displaystyle X^{n+1}_{t}=x+&\int_{0}^{t}\Big(f_s(X^{n+1}_{s},Y^{n+1}_{s},Z^{n+1}_{s},K^{n+1}_{s})-\delta(Y^{n+1}_{s}-Y^{n}_{s})\Big)ds\nonumber\\
&\displaystyle +\int_{0}^{t}\Big(\sigma_s(X^{n+1}_{s},Y^{n+1}_{s},Z^{n+1}_{s},K^{n+1}_{s})-\delta(Y^{n+1}_{s}-Y^{n}_{s})\Big)dW_{s} \nonumber\\&+\int_{0}^{t}\int_{E}\Big(\beta_s(X^{n+1}_{s},Y^{n+1}_{s},Z^{n+1}_{s},K^{n+1}_{s})-\delta(K^{n+1}_{s}-K^{n}_{s})\Big)\tilde{\pi}(ds,de),\end{align}
\begin{align*}
&\displaystyle Y_{t}^{n+1}=g(X^{n+1}_{T})-\int_{t}^{T}h_s(X^{n+1}_{s},Y^{n+1}_{s},Z^{n+1}_{s},K^{n+1}_{s})ds-\int_{t}^{T}Z^{n+1}_{s}dW_{s}-\int_{t}^{T}\int_{E}K^{n+1}_{s}(e)\tilde{\mu}(ds,de)
\end{align*}
For $n\geq 1$, $t \in[0,T]$, we consider the following processes
$$
\hat{X}^{n+1}_{t}:=X^{n+1}_{t}-X^{n}_{t},\quad \hat{Y}^{n+1}_{t}:=Y^{n+1}_{t}-Y^{n}_{t},\quad \hat{Z}^{n+1}_{t}:=Z^{n+1}_{t}-Z^{n}_{t},\qquad \hat{K}^{n+1}_{t}:=K^{n+1}_{t}-K^{n}_{t}.
$$ 
and for a function $\phi=\{f,h,\sigma,\beta\}$, we set 
$$
\hat{\phi}^{n+1}_{t}:=\phi(t,U^{n+1}_{t},\nu^{n}_{t})-\phi(t,U^{n}_{t},\nu^{n-1}_{t}), \quad \tilde{\phi}^{n}_{t}:=\phi(t,U^{n}_{t},\nu^{n}_{t})-\phi(t,U^{n}_{t},\nu^{n-1}_{t}) .
$$
In order to prove the existence of the solution, we will show that $(X^{n}, Y^{n},Z^{n},K^{n})_{n\geq 0}$ is a Cauchy sequence. \\
First, we apply Ito's formula to $\hat{ X}^{n+1}\hat{Y}^{n+1}$ and take the expectation  
\begin{align*}
&\E[\hat{X}_{T}^{n+1}\hat{Y}_{T}^{n+1}]=\E[\int_{0}^{T}\hat{Y}^{n+1}_{s}[\hat{f}^{n+1}_{s}-\delta (\hat{Y}^{n+1}_{s}-\hat{Y}^{n}_{s})]ds+\E[\int_{0}^{T}\hat{X}^{n+1}_{s}\hat{h}^{n+1}_{s}ds]\nonumber\\
&+\E[\int_{0}^{T}(\hat{\sigma}^{n+1}_{s}-\delta(\hat{Z}^{n+1}_{s}-\hat{Z}^{n}_{s}),\hat{Z}^{n+1}_{s})ds+\int_{0}^{T}\int_{E}\hat{K}^{n+1}_{s}(\hat{\beta}^{n+1}_{s}-\delta(\hat{K}^{n+1}_{s}-\hat{K}^{n}_{s}))\eta(de,ds)].
\end{align*}
Rearranging terms, we get 
\begin{align}\label{Appendix Ito}
&\E[\hat{X}^{n+1}_{T}(g(X^{n+1}_{T})-g(X^{n}_{T}))]+\delta\E[\int_{0}^{T}|\hat{Y}^{n+1}_{s}|^{2}+\|\hat{Z}^{n+1}_{s}\|^{2}+|\hat{K}^{n+1}_{s}|_{s}^{2}ds]\nonumber\\
-&\E\big[\int_{0}^{T}\hat{X}^{n+1}_{s}\hat{h}^{n+1}_{s}+\hat{Y}^{n+1}_{s}\hat{f}^{n+1}_{s}+\hat{\sigma}^{n+1}(s)\hat{Z}^{n+1}_{s}ds+\int_{0}^{T}\int_{E}\hat{K}^{n+1}_{s}\hat{\beta}^{n+1}_{s}\eta(de,ds)\big]\nonumber\\
=\delta &\E[\int_{0}^{T}\hat{Y}^{n+1}_{s}\hat{Y}^{n}_{s}+\hat{Z}^{n+1}_{s}\hat{Z}^{n}_{s}ds+\int_{0}^{T}\int_{E}\hat{K}^{n+1}_{s}\hat{K}^{n}_{s}\eta(de,ds)\big].
\end{align}
Using Assumption \eqref{bar H1} we get 
\begin{align}\label{ItoH}
&\E\big[ k^{'}|\hat{X}^{n+1}_{T}|^{2}+\delta\int_{0}^{T}|\hat{Y}^{n+1}_{s}|^{2}+\|\hat{Z}^{n+1}_{s}\|^{2}+|\hat{K}^{n+1}_{s}|_{s}^{2}ds
+k\int_{0}^{T}|\hat{X}^{n+1}_{s}|^{2}ds\big]\nonumber\\
&\leq\delta\E\big[\int_{0}^{T}\hat{Y}^{n+1}_{s}\hat{Y}^{n}_{s}+\hat{Z}^{n+1}_{s}\hat{Z}^{n}_{s}ds+\int_{0}^{T}\int_{E}\hat{K}^{n+1}_{s}\hat{K}^{n}_{s}\eta(de,ds)\big].
\end{align}
In addition, the elementary inequality $ab\leq a^{2}/2 +b^{2}/2$  yields to 
\begin{align}\label{EstimatesH}
&\E[\int_{0}^{T}\hat{Y}^{n+1}_{s}\hat{Y}^{n}_{s}+\hat{Z}^{n+1}_{s}\hat{Z}^{n}_{s}ds]+\int_{0}^{T}\int_{E}\hat{K}^{n+1}_{s}\hat{K}^{n}_{s}\,\eta(de,ds)\nonumber\\
&\leq \frac{1}{2}\E[\int_{0}^{T}|\hat{Y}^{n+1}_{s}|^{2}+\|\hat{Z}^{n+1}_{s}\|^{2}+| \hat{K}^{n+1}_{s}|_{s}^{2}ds]\nonumber\\
&+\frac{1}{2}\E[\int_{0}^{T}|\hat{Y}^{n}_{s}|^{2}+\|\hat{Z}^{n}_{s}\|^{2}+| \hat{K}^{n}_{s}|_{s}^{2}ds]
\end{align}
Pugging \eqref{EstimatesH} in \eqref{ItoH}, we obtain 
\begin{align}\label{estim F}
&k^{'}\E[|\hat{X}^{n+1}_{T}|^{2}]+k\E[\int_{0}^{T}|\hat{X}^{n+1}_{s}|^{2}]+\frac{\delta}{2}\E[\int_{0}^{T}|\hat{Y}^{n+1}_{s}|^{2}+\|\hat{Z}^{n+1}_{s}\|^{2}+|\hat{K}^{n+1}_{s}|_{s}^{2}ds]\\
&\leq  \frac{ \delta}{2}\left[  \E[\int_{0}^{T}|\hat{Y}^{n}_{s}|^{2}+\|\hat{Z}^{n}_{s}\|^{2}+|\hat{K}^{n}_{s}|_{s}^{2}ds]\right]
\end{align}
\underline{Step2:}\quad$\forall n\geq 1, \E\left[\int_{0}^{T}|\hat{Y}^{n}_{s}|^{2}+\|\hat{Z}_{s}\|^{2}+|\hat{K}|_{s}^{2}ds\right]\leq C^{1}\Big(\E\left[|\hat{X}^{n}_{T}|^{2}+\int_{0}^{T}|\hat{X}^{n}_{s}|^{2}ds\right]\Big)$\\
Apply It\^o formula to $|\hat{Y}^{n}|^{2}$
\begin{align*}
\displaystyle{|\hat Y^{n}_{t}|^{2}}&+\displaystyle{\int_{t}^{T}\|\hat Z^{n}_{s}\|^{2}ds+\int_{t}^{T} \int_{E}|\hat K^{n}_{s}(e)|^{2}\eta(de,ds)=|\hat Y^{n}_{T}|^{2}+}\displaystyle{2 \int_{t}^{T}\hat Y^{n}_{s}\hat{h}^{n}(s)ds}\\
&\displaystyle{-2\int_{t}^{T} \int_{E} \hat Y^{n}_{s}\hat K^{n}_{s}(e)\tilde{\pi}(de,ds)}-\displaystyle{2\int_{t}^{T}\hat Y^{n}_{s}\hat Z^{n}_{s}dB_{s}}
\end{align*}
Taking the expectation with Assumption, we obtain that 
\begin{align*}
\E\displaystyle{|\hat Y^{n}_{t}|^{2}+\frac{1}{2}\E\int_{t}^{T}\| \hat Z^{n}_{s}\|^{2}ds+}&\frac{1}{2}\E\int_{t}^{T} \int_{E}|\hat K^{n}_{s}(e)|^{2}\eta(de,ds)\\
&\leq\E| \hat Y^{n}_{T}|^{2}+\epsilon\displaystyle{ \E\int_{t}^{T}| \hat Y^{n}_{s} |^{2}ds}+\epsilon\displaystyle{ \E\int_{t}^{T}|\hat X^{n}_{s}|^{2}ds}.
\end{align*}
From there we can write 
\begin{equation*}
\E\displaystyle{|\hat Y^{n}_{t}|^{2}}\leq \tilde{C}\left[\displaystyle{ \E\int_{t}^{T}|\hat Y^{n}_{s} |^{2}ds}+\displaystyle{ \E\int_{t}^{T}|\hat X^{n}_{s} |^{2}ds}\right],\hspace{0.5cm} 0\leq t\leq T 
\end{equation*}
and hence, Gronwall's lemma yield's to 
\begin{equation}
\E\displaystyle{|\hat Y^{n}_{t}|^{2}}\displaystyle{ \leq\tilde{C}\left[\E [|g(X^{n+1}_{T})-g(X^{n}_{T})|^{2}]+\E\int_{0}^{T}|\hat X^{n}_{s} |^{2}ds\right]}.
\end{equation}
Then
\begin{align}
\E[\displaystyle{\int_{0}^{T}|\hat Z^{n}_{s}|^{2}ds+\int_{0}^{T} \int_{E}|\hat K^{n}_{s}(e)|^{2}\eta(de,ds)}]&\leq\E\displaystyle{|\hat Y^{n}_{t}|^{2}}+\displaystyle{\int_{0}^{T}|\hat Z^{n}_{s}|^{2}ds+\int_{0}^{T} \int_{E}|\hat K^{n}_{s}(e)|^{2}\eta(de,ds)}\nonumber\\
&\leq\tilde{C}[\E|\hat Y^{n}_{T}|^{2}+\displaystyle{ \E\int_{0}^{T}|\hat Y^{n}_{s} |^{2}ds}+\displaystyle{ \E\int_{0}^{T}|\hat X^{n}_{s} |^{2}ds}]
\end{align}
Finally, since $g$ is Lipschitz with respect to $x$ we have  
\begin{align}
\displaystyle{ \E\int_{0}^{T}|\hat Y^{n}_{s} |^{2}ds}+\frac{1}{2}\E[\displaystyle{\int_{0}^{T}|\hat Z^{n}_{s}|^{2}ds+\int_{0}^{T} \int_{E}|\hat K^{n}_{s}(e)|^{2}\eta(de,ds)}]&\leq\tilde{C}[\E|\hat Y^{n}_{T}|^{2}+\displaystyle{ \E\int_{0}^{T}|\hat X^{n}_{s} |^{2}ds}]\nonumber\\
&\leq\tilde{C}[\E|\hat X^{n}_{T}|^{2}+\displaystyle{ \E\int_{0}^{T}|\hat X^{n}_{s} |^{2}ds}].\nonumber
\end{align}
Plugging this estimates in \eqref{estim F}, we obtain
\begin{align}\label{Appendix estim F}
& \min(k,k^{'})\,\E[|\hat{X}^{n+1}_{T}|^{2}]+\E[\int_{0}^{T}|\hat{X}^{n+1}_{s}|^{2}]\leq  \frac{ \delta \tilde{C}}{2}\big[\E|\hat X^{n}_{T}|^{2}+\displaystyle{ \E\int_{0}^{T}|\hat X^{n}_{s} |^{2}ds}\big].
\end{align}
Take $\delta:=min(k,k^{'})/\tilde{C}$  we obtain that 
$$ \E[|\hat{X}^{n+1}_{T}|^{2}]+\E[\int_{0}^{T}|\hat{X}^{n+1}_{s}|^{2}]\leq  \frac{ 1}{2}\big[\E|\hat X^{n}_{T}|^{2}+\displaystyle{ \E\int_{0}^{T}|\hat X^{n}_{s} |^{2}ds}\big]
$$
Henceforth, $(\hat{X}^{n}_{T})_{n\geq 0}$ is a Cauchy sequence in $\mathcal{L}^{2}(\Omega,\P)$ and $(\hat{X}^{n})_{n\geq 0},(\hat{Y}^{n})_{n\geq 0}, (\hat{Z}^{n})_{n\geq 0}$ and $\hat{K}^{n})_{n\geq 0})$ are Cauchy sequences respectively in $\mathcal{L}^{2}([0,T],\Omega,dt\otimes d\P)$and $\mathcal{L}^{2}_{\nu}([0,T],\Omega,dt\otimes d\nu).$ Hence, if $X, Y$, $Z$ and $K$ are the respective limits of these sequences, passing to the limit in \eqref{sequence FBSDE}, we see that $(X,Y,Z,K)$ is a solution of
\begin{align*}
&X_{t}=X_{0}+\displaystyle\int_{0}^{t}f(s,X_{s},Y_{s},Z_{s},K_{s}(x))ds+\int_{0}^{t}\sigma(s,X_{s},Y_{s},Z_{s},K_{s}(x)))dW_{s}\\
\vspace{0.3cm}
& \qquad +\int_{0}^{t}\int_{E}\beta(s,X_{s^{-}},Y_{s^{-}},Z_{s^{-}},K_{s}(x))\tilde{\pi}(ds,de),\hspace{0.2cm}0 \leq t\leq T,\, \P-a.s.\\
&Y_{t}=g(X_{T},\mathbb{P}_{X_{T}})-\int_{t}^{T}h(s,X_{s},Y_{s},Z_{s},K_{s}(e))ds-\int_{t}^{T}Z_{s}dW_{s}-\int_{t}^{T}\int_{E}K_{s}(x)\tilde{\pi}(ds,de),
\end{align*}
\underline{Step3: Uniqueness} To show that the system \eqref{Appendix system} has a unique solution. we suppose that $(\bar{X},\bar{Y},\bar{Z},\bar{K})$ is also solution of \eqref{Appendix system}. Let $\bar
X=X^{1}-X^{2}$, $\bar
Y=Y^{1}-Y^{2}$, $\bar
X=Z^{1}-Z^{2}$, $\bar
K=K^{1}-K^{2}$. Then with Itô’s formula, we obtain \\
\begin{align*}
&d\bar X_{T}\bar Y_{T}=\bar X_{t}d\bar Y_{t}+\bar Y_{t}d\bar X_{t}+d\langle\bar X,\bar{Y}\rangle_{t}.
\end{align*}
Hence 
\begin{align*}
\E[\bar X_{T}\bar Y_{T}]&=\E[\int_{0}^{T}\bar{Y}_{s}\big[b_{s}(X^{1}_{s},Y^{1}_{s},Z^{1}_{s},K^{1}_{s})-f_{s}(X_{s}^{2},Y_{s}^{2},Z_{s}^{2},K_{s}^{2})\big]ds]\\
&+\E[\int_{0}^{T}\bar{X}_{s}\big[h_{s}(X^{1}_{s},Y^{1}_{s},Z^{1}_{s},K^{1}_{s})-b_{s}(X_{S}^{2},Y_{s}^{2},Z_{s}^{2},K_{s}^{2})\big]ds]\\
&+\E[\int_{0}^{T}\bar{Z}_{s}\big[\sigma_{s}(X^{1}_{s},Y^{1}_{s},Z^{1}_{s},K^{1}_{s})-\sigma_{s}(X_{s}^{2},Y_{s}^{2},Z_{s}^{2},K_{s}^{2})ds\big]\\
&+\int_{0}^{T}\int_{E}\bar{K}_{s}(e)\big[\beta_{s}(X^{1}_{s},Y^{1}_{s},Z^{1}_{s},K^{1}_{s})-\beta_{s}(X^{2},Y_{s}^{2},Z_{s}^{2},K_{s}^{2})\big]\eta(de,ds)].\\
\end{align*}
Using Assumption $(H1)-(ii)$, we have
\begin{equation}
[k^{'}|\bar{X}_{T}|^{2}] \leq \E[\bar X_{T}\bar Y_{T}].
\end{equation}
and the right side of the above it\^o equation is in fact $\mathcal{\bar A}(t,U^{1},U^{2})$ and hence using once again Assumption $(H1)-(i)$, we get 
 \begin{align*}
 k^{'}\E[|\bar{X}_{T}|^{2}]\leq -k\E[\int_{0}^{T}|\bar{X}_{s}|^{2}ds\qquad\iff k^{'}\E[\bar{X}_{T}|^{2}]+k\E[\int_{0}^{T}|\bar{X}_{s}|^{2}ds]\leq 0.
 \end{align*}
 We conclude that $ \bar{X}=0,\,\,\P$-a.s. and it follow that $\bar{Y}=0$.
\end{proof}
 
\bibliographystyle{acm}
\bibliography{Extended-MFG-18}

\end{document}